\documentclass[10pt]{amsart}
\usepackage{amssymb,amsfonts,lineno,amsthm,amscd,stmaryrd}
\usepackage[leqno]{amsmath}
\usepackage[numbers]{natbib}

\theoremstyle{plain}
\newtheorem{thm}{Theorem}
\newtheorem{cor}{Corollary}
\newtheorem{lem}[cor]{Lemma}
\newtheorem{prop}[cor]{Proposition}
\theoremstyle{definition}
\newtheorem{definition}[cor]{Definition}
\newtheorem{remark}[cor]{Remark}
\newtheorem{example}[cor]{Example}

\numberwithin{cor}{section}
\numberwithin{equation}{section}

\newcommand{\EQ}[1]{\eqref{eq:#1}}
\newcommand{\LEM}[1]{Lemma~\ref{lem:#1}}

\newcommand{\THM}[1]{Theorem~\ref{thm:#1}}
\newcommand{\REM}[1]{Remark~\ref{rem:#1}}
\newcommand{\PROP}[1]{Proposition~\ref{prop:#1}}
\newcommand{\COR}[1]{Corollary~\ref{cor:#1}}
\newcommand{\SEC}[1]{Section~\ref{sec:#1}}

\newcommand{\ENUM}[1]{(\ref{enum:#1})}
\newcounter{hypo}

\DeclareMathOperator{\trace}{trace}

\DeclareMathOperator{\sign}{sign}
\DeclareMathOperator{\divg}{div}

\newcommand{\R}{\ensuremath{\mathbb{R}}}
\newcommand{\rn}{\R^n}

\newcommand{\M}{\ensuremath{\mathbb{M}}}
\newcommand{\Prob}[1]{\ensuremath{\mathbb{P}\left[ {#1} \right]}}

\newcommand{\iden}{\ensuremath{I_n}}
\newcommand{\ep}{\varepsilon}
\newcommand{\Sy}{\ensuremath{\mathcal{S}_n}}
\newcommand{\puccisub}[2]{\mathcal{P}^-_{#1,#2}}
\newcommand{\Puccisub}[2]{\mathcal{P}^+_{#1,#2}}
\newcommand{\PucciSub}{\Puccisub{\elp}{\Elp}}
\newcommand{\pucciSub}{\puccisub{\elp}{\Elp}}
\newcommand{\pucci}{\mathcal{P}^-}
\newcommand{\Pucci}{\mathcal{P}^+}

\newcommand{\rescale}[2]{\mathcal{T}^{#1}_{#2}}

\newcommand{\homcls}[1]{H_{#1}}
\newcommand{\homclsplus}[1]{\homcls{#1}^+}
\newcommand{\Elp}{\Lambda}
\newcommand{\elp}{\lambda}
\newcommand{\anom}{{\alpha^*}}

\newcommand{\radfun}[1]{\xi_{#1}}
\newcommand{\Rnpunct}{\R^n \setminus \{ 0 \}}

\newcommand{\mrbox}[1]{\quad \mbox{#1} \ }

\begin{document}
\title[Fundamental solutions of fully nonlinear elliptic equations]{Fundamental solutions of homogeneous\\ fully nonlinear elliptic equations}
\author{Scott N. Armstrong}
\address{Department of Mathematics\\ Louisiana State University\\ Baton Rouge, LA 70803.}
\email{armstrong@math.lsu.edu}
\author{Boyan Sirakov}
\address{UFR SEGMI, Universit\'e Paris 10\\
92001 Nanterre Cedex, France \\
and CAMS, EHESS \\
54 bd Raspail \\
75270 Paris Cedex 06, France}
\email{sirakov@ehess.fr}
\author{Charles K. Smart}
\address{Department of Mathematics\\
University of California\\ Berkeley, CA 94720.}
\email{smart@math.berkeley.edu}
\date{\today}

\begin{abstract}
We prove the existence of two fundamental solutions $\Phi$ and $\tilde \Phi$ of the PDE
\begin{equation*}
F(D^2\Phi) = 0 \quad \mbox{in} \ \R^n \setminus \{ 0 \}
\end{equation*}
for any positively homogeneous, uniformly elliptic operator $F$. Corresponding to $F$ are two unique scaling exponents $\alpha^*, \tilde\alpha^* > -1$ which describe the homogeneity of $\Phi$ and $\tilde \Phi$. We give a sharp characterization of the isolated singularities and the behavior at infinity of a solution of the equation $F(D^2u) = 0$, which is bounded on one side. A Liouville-type result demonstrates that the two fundamental solutions are the unique nontrivial solutions of $F(D^2u) = 0$ in $\R^n \setminus \{ 0 \}$ which are bounded on one side in both a neighborhood of the origin as well as at infinity. Finally, we show that the sign of each scaling exponent is related to the recurrence or transience of a stochastic process for a two-player differential game.
\end{abstract}

\maketitle


\section{Introduction and main results}

The following fundamental result was proved by M. B\^ocher in 1903.
\begin{thm}[B\^ocher, \cite{bocher}]\label{thm:bocher} Denote $B_r = \{ x \in \R^n: |x| < r\}$, and assume $n\geq 2$.

(i) Suppose $u \in C(B_1 \setminus \{ 0\})$ is harmonic and bounded below or above in $B_1\setminus\{0\}$. Then either $u$ can be extended to a harmonic function in $B_1$, or there exist constants $a\not=0$ and $C> 0$ such that
\begin{equation*}
a \Phi -C \leq u \leq a\Phi + C \quad \mbox{in} \ B_{1/2} \setminus \{ 0 \},\quad \mbox{where}\quad \Phi(x)=\left\{\begin{array}{lcl}
|x|^{2-n} &\mbox{if}& n>2\\
-\log |x| &\mbox{if}&n=2.
\end{array}\right.
\end{equation*}
Hence, by the linearity of the Laplacian, $u-a\Phi$ can be extended to a harmonic function in $B_1$.

(ii) Suppose $u$ is harmonic and bounded below or above in $\rn\setminus B_1$, $n\ge3$. Then $u(x)\to a_0$ as $|x|\to\infty$, for some $a_0\in \R$.
\end{thm}

The function $\Phi$ which appears in this theorem  is known as {\it the fundamental solution} for the Laplacian. B\^{o}cher's result easily implies  the following extended  Liouville theorem.
\begin{thm}[B\^ocher, \cite{bocher}]\label{thm:bocher2} The set of all harmonic in $\R^n\setminus\{0\}$ functions that are  bounded from above or from below in a neighbourhood of zero as well as in a neighbourhood of infinity is in the form $\{ a\Phi + b\;|\; a,b\in\R\}$.\end{thm}

In this article we construct {fundamental solutions} of the fully nonlinear equation
\begin{equation}\label{eq:princ}
F(D^2u) = 0,
\end{equation}
and extend \THM{bocher} to solutions of \EQ{princ}, under the assumption that $F$ is a uniformly elliptic and positively homogeneous operator. Here equation \EQ{princ} is understood in the viscosity sense (c.f. \cite{UsersGuide,Caffarelli:Book}). We recall that, in general, the best regularity available for a solution $u$ of equation \EQ{princ} is $u \in C^{1,\gamma}_{\mathrm{loc}}$, for some constant $\gamma >0$ which depends on the operator $F$ (see \cite{Caffarelli:Book,Trudinger,Nadirashvili}).

Before proceeding to the precise statements of our results, let us give some additional context. An extension of \THM{bocher}  to some linear equations appeared already in \cite{bocher}, while a thorough study of fundamental solutions and isolated singularities of linear equations, in view of more modern theories, was performed  by Gilbarg and Serrin \cite{GilbargSerrin}. Later, in a sequence of papers, Serrin \cite{Serrin1,Serrin2,Serrin3} produced a deep study of singular solutions of general quasilinear divergence-form equations
$
-\divg A(x,u,Du) = B(x,u,Du),
$ $p$-harmonic functions being the model case.  We refer to \cite{Veron,Nicolosi} for more  developments and references on solutions of quasilinear equations. We also refer to \cite{Miranda} for more on the existence of fundamental solutions of linear and quasilinear equations.

In recent years, there have been a number of studies of singular solutions of the fully nonlinear equation \EQ{princ}, in the particular case when $F$ is a rotationally invariant operator, that is, $F(D^2u)$ depends only on the eigenvalues of $D^2u$. The work most closely related to ours is the one by Labutin \cite{Labutin:2001}, who gave, among other things, a partial extension of Bocher's theorem to solutions of Pucci extremal equations. Below we discuss in more detail that paper and the additional hypotheses it involved. We also note that in the last several years there has been a great amount of interest of singular solutions of conformally invariant fully nonlinear equations (we refer to \cite{Li1,Li2,CafLiNir} and the references in these works).

The essence of all results on isolated singularities is that if a function fails to be a solution at an isolated  point and is bounded on one side in a neighbourhood of this point, then $u$ behaves like a {\it fundamental solution} of the elliptic operator near the isolated point. In the literature the term fundamental solution usually refers to a solution in $\R^n$ (or in some domain of interest) except at zero, which goes to infinity at the origin and is bounded away from it. We are going to use this term also for solutions in $\R^n\setminus \{0\}$ that have the inverse behavior, that is, are bounded on bounded sets and tend to  infinity at infinity. For instance, the fundamental solution for the $p$-Laplace equation is $\Phi_p(x)=|x|^{(p-n)/(p-1)}$ if $ p\not= n$. In \cite{Serrin1,Serrin2} only the case $p<n$ was considered, but as remarked for instance in \cite{KichenassamyVeron}, similar asymptotics as in \THM{bocher}(i) hold if $p>n$.

Let us now state our main results. We consider an arbitrary \emph{Isaacs operator}, that is, a nonlinear map $F$ from the set $\Sy$ of $n$-by-$n$ real symmetric matrices into $\R$, with the following two properties.
 \begin{itemize}
 \item[(H1)] $F$ is uniformly elliptic and Lipschitz: for some constants $0<\lambda\le \Lambda$ and all real symmetric matrices $M$ and $N$, with $N$ nonnegative definite, we have
     \begin{equation*}
    \lambda\trace(N) \le F(M-N)-F(M) \leq \Lambda\trace(N).
    \end{equation*}
\item[(H2)] $F$ is positively homogeneous of degree 1: \[ F(tM)= tF(M) \quad \mbox{for each} \  t\ge 0 \ \mbox{ and } \mbox{ each } \mbox{ symmetric }\ M.\]
 \end{itemize}
  We emphasize that (H1) and (H2) will be the only hypotheses on $F$. In particular, we assume neither that $F$ is convex or concave, nor that $F$ is rotationally invariant. In can be shown that (H1)-(H2) are equivalent to
\begin{equation}\label{eq:flip}
\displaystyle F(D^2u)=\sup_{\alpha\in \mathcal{A}}\inf_{\beta\in \mathcal{B}} \left(-a^{\alpha,\beta}_{ij}\partial_{ij}u \right) \quad \mbox{ or }\quad F(D^2u)=\inf_{\alpha\in \mathcal{A}}\sup_{\beta\in \mathcal{B}} \left(-a^{\alpha,\beta}_{ij}\partial_{ij}u\right)
\end{equation}
for index sets $\mathcal{A}$, $\mathcal{B}$ and symmetric matrices $A^{\alpha,\beta}=(a^{\alpha,\beta}_{ij})$, with $\lambda I\le A^{\alpha,\beta}\le \Lambda I$.

In the following theorem we establish the existence and the main properties of the fundamental solution.
\begin{thm}\label{thm:fundysol}
There exists a non-constant solution  of \EQ{princ} in $\R^n \setminus \{0\}$
that is bounded below in $B_1$ and bounded above in $\R^n \setminus B_1$. Moreover, the set of all such solutions is of the form $\{a\Phi + b\;|\; a>0, b\in \R\}$, where  $\Phi\in C^{1,\gamma}_{\mathrm{loc}}(\R^n \setminus \{ 0 \})$ can be chosen to satisfy one of the following homogeneity relations: for all $t>0$ 
\begin{equation}\label{eq:scale}
{\Phi}(x) = {\Phi}(tx) + \log t\qquad\mbox{ or }\qquad
\left\{\begin{array}{l} {\Phi}(x) =t^{\anom} {\Phi}(tx)\\
\anom\Phi(x)>0\end{array}\right.\quad\mbox{ in }\;  \R^n \setminus \{0\}
,
\end{equation}
for some number $\anom \in (-1,\infty)\setminus\{0\}$ which depends only on $F$.
\end{thm}

\begin{definition}  We call the number $\anom=\anom(F)$ the \emph{scaling exponent} of $F$, and we set $\anom(F)=0$ in the case the first alternative in \EQ{scale} occurs.
\end{definition}

\begin{definition}
 We call the function $\Phi $ whose existence is asserted in \THM{fundysol}, and  normalized so that
\begin{equation}\label{eq:fundysol-normalize}
\min_{\partial B_1} \left(\sign(\anom)\Phi\right) =1 \quad\mathrm{if}\; \anom \neq 0,\quad \mbox{and} \quad \int_{ \partial B_1 } \Phi=0\quad\mathrm{if}\; \anom = 0,
\end{equation}
the \emph{upward-pointing fundamental solution} of \EQ{princ}.
\end{definition}

\begin{remark} By \EQ{scale} and \EQ{fundysol-normalize}, the upward-pointing fundamental solution is strictly decreasing in the radial direction, and we have
\begin{gather*}
\lim_{|x|\to0}\Phi(x) = \infty \quad \mbox{if} \ \anom(F)\ge0,  \quad \mbox{and} \quad \lim_{|x|\to\infty}\Phi(x) = 0\quad \mbox{if} \ \anom(F)>0,\\
\lim_{|x|\to0} \Phi(x) = 0 \quad \mbox{if} \ \anom(F)<0,\quad \mbox{and} \quad \lim_{|x|\to\infty}\Phi(x) = -\infty \quad \mbox{if} \  \anom(F)\le0.
\end{gather*}
\end{remark}

For any $F$ satisfying (H1)-(H2) we denote the dual operator $\tilde F$ of $F$ by
\begin{equation*}
\tilde F(M):=-F(-M).
\end{equation*}
 Notice that the two operators appearing in \EQ{flip} are dual in this sense, as are the Pucci extremal operators. By \THM{fundysol}, the operator $\tilde F$ has an upward-pointing fundamental solution which we denote by $\tilde \Phi$. It follows that the function $-\tilde \Phi$ is another solution of $F(D^2u) = 0$ in $\R^n\setminus \{ 0 \}$, and we call it the \emph{downward-pointing fundamental solution} of~$F$.

 The following result  shows $\Phi$ and $-\tilde \Phi$ are the only fundamental solutions of \EQ{princ}, and extends \THM{bocher2} to fully nonlinear operators.

 \begin{thm}\label{thm:liouville} Suppose that $u \in C(\R^n\setminus \{ 0 \})$ is a solution of the equation
\begin{equation*}
F\left( D^2u \right) = 0 \quad \mbox{in} \ \R^n \setminus \{ 0 \}.
\end{equation*}
Suppose further that $u$ is bounded from above or below in $B_1 \setminus \{ 0 \}$, and that $u$ is bounded from above or below in $\R^n \setminus B_1$.  Then either $u \equiv b$, or $u \equiv a \Phi + b$, or $u \equiv - a\tilde \Phi + b$ for some $a > 0$, $b \in \R$.
\end{thm}

Prior to this paper, the existence of a fundamental solution of a fully nonlinear equation was known only for certain rotationally invariant operators, when a direct calculation verifies that the function $\radfun{\alpha}$ defined by
\begin{equation}\label{eq:radfun}
\radfun\alpha(x) := \left\{ \begin{array}{ll}
|x|^{-\alpha} & \mathrm{if} \ \ \alpha > 0, \\
- \log |x| & \mathrm{if} \ \ \alpha = 0, \\
- |x|^{-\alpha} & \mathrm{if} \ \ \alpha < 0,
\end{array}\right.
\end{equation}
satisfies the equation the some $\alpha > -1$ (see \cite{Cutri:2000,Labutin:2001,Felmer:2009}). For example, a direct computation shows that the upward-pointing fundamental solutions of the Pucci extremal operators $\mathcal{P}^-_{\lambda,\Lambda}$ and $\mathcal{P}^+_{\lambda,\Lambda}$ are $\radfun{\lambda(n-1)/\Lambda -1}$ and $\radfun{\Lambda(n-1)/\lambda -1}$, respectively. It follows from the uniqueness result above that the upward-pointing fundamental solution of a rotationally invariant operator must be $\radfun{\alpha}$ for some $\alpha > -1$.

It should be noted that, by well-known results from the theory of linear elliptic PDE, if $\Phi$ is the fundamental solution of a linear operator $L$, then $L\Phi$ is interpreted as the Dirac mass at the origin. For fully nonlinear operators this is not true, as noticed by Labutin \cite{Labutin:2001} who showed that if $\lambda\neq\Lambda$, then $\mathcal{P}^+_{\lambda,\Lambda}\left(\radfun{\Lambda(n-1)/\lambda -1} \right)$ vanishes near the origin in a reasonable weak sense.

The Liouville-type \THM{liouville}  is, to our knowledge, the first result of its kind for fully nonlinear operators, and in particular is new even for the Pucci extremal operators. Of course, even in the case $F(D^2u) = -\Delta u$ we cannot relax the hypothesis that $u$ be bounded on one side in both a neighborhood of the origin and a neighborhood of infinity. This we recall by considering the function $u(x) = x_1^2 - x_2^2$ and its Kelvin transform $v(x) = |x|^{-n-2} (x_1^2 - x_2^2)$, both of which are harmonic in $\R^n \setminus \{ 0 \}$. These functions also show that we may not relax the hypotheses on the solution in our theorems classifying isolated singularities below.

\begin{remark} \label{rem:anom}
Informally, the scaling exponents $\anom(F)$ and $\anom(\tilde F)$ characterize the intrinsic \emph{internal scalings} of the operator $F$, and we think of each scaling exponent as a kind of principal eigenvalue of a certain elliptic equation on the unit sphere.  Indeed, several of the ideas we employ in our proof of \THM{fundysol} are related to the principal eigenvalue theory for fully nonlinear operators developed in \cite{Lions:1983,Birindelli:2006,Quaas:2008,Armstrong:2009}. As we will see, $\anom(F)$ is given by
\begin{multline*}
\alpha^*(F) = \sup\left\{\alpha\in (-1,\infty) \setminus\{ 0 \} : \ \mbox{there exists an} \ \mbox{$(-\alpha)$--homogeneous} \right. \\  \left. \mbox{supersolution  of} \ F(D^2v) \geq 0 \ \mbox{such that} \ \alpha v>0 \  \mbox{in}\ \R^n \setminus \{ 0 \}  \right\},
\end{multline*}
and satisfies
\begin{equation*}
-1< \frac{\lambda}{\Lambda} (n-1) -1 \leq \alpha^*(F) \leq \frac{\Lambda}{\lambda} (n-1) -1.
\end{equation*}
Of course, if $F$ is a linear operator, then $\alpha^*(F)=n-2$. In  \SEC{examples} we discuss more properties of the scaling exponents. Furthermore, in  \SEC{stoch} we will see that $\anom(F)$ has an interesting stochastic interpretation. Corresponding to the Isaacs operator is a diffusion process controlled by two competing players. The sign $\anom(F)$ indicates whether the first player can force the diffusion to return to the origin, or whether the second player can force the process out to infinity (almost surely). In the case $\anom(F)=0$, the diffusion is \emph{recurrent} (that is, it returns infinity many times to every neighborhood of the origin almost surely) but returns to the origin with zero probability. This generalizes the well-known fact that Brownian motion is recurrent in dimensions $n=1,2$, and \emph{transient} in dimensions $n\geq 3$.
\end{remark}

The principal difficulty in the proof of \THM{fundysol} is establishing the existence of a fundamental solution. We first define $\anom(F)$ according to the formula in \REM{anom}, and show that $F$ satisfies a maximum principle with respect to $(-\alpha)$--homogeneous functions, for $\alpha < \anom(F)$. The rest of our argument is quite different for the cases $\anom(F)>0$ and $\anom(F)\leq 0$. In the first case, we use a construction based on comparison principle and the Perron method, while in the second case we appeal to an abstract topological fixed point theorem, which helps us to build approximate fundamental solutions. As we will see, this difference is due to the fact that the comparison principle is reversed on the space of $(-\alpha)$-homogeneous functions, for $\alpha < 0$, while the regular comparison principle holds for $\alpha >0$ (see \PROP{HCP} below).

\medskip

As an application of \THM{fundysol}, we are able to completely characterize the isolated singularities of solutions of $F(D^2u) = 0$ that are bounded on one side in a neighborhood of the singularity. For brevity we introduce  the following notation:
 \begin{equation*} u\underset{0}{\sim}v\;\mbox{ if }\; \frac{u(x)}{v(x)} \to a \; \mbox{ as }\; x\to 0, \quad  u\underset{0}{\approx}v\;\mbox{ if }\; a v -C \leq u \leq av + C\; \mbox{ in } \; B_{1/2}\setminus\{0\},\end{equation*} for some $a,C>0$, and similarly if in these formulas $0$ is replaced by $\infty$ and $B_{1/2}\setminus\{0\}$ is replaced by $\R^n\setminus B_2$. When we write $u\underset{0}{\sim}v$ (resp. $u\underset{0}{\approx}v$; $u\underset{\infty}{\sim}v$; $u\underset{\infty}{\approx}v$), it is to be understood that $u,v\underset{x\to0}{\to}0$ (resp. $u,v\underset{x\to0}{\to}\infty$;
 $u,v\underset{|x|\to\infty}{\longrightarrow}0$; $u,v\underset{|x|\to\infty}{\longrightarrow }-\infty$).

\begin{thm}\label{thm:singularities}
Suppose $u \in C(B_1 \setminus \{ 0 \})$ is a viscosity solution of the equation
\begin{equation}
F(D^2u) = 0 \quad \mbox{in} \ B_1 \setminus \{ 0 \}.
\end{equation}
If $u$ is bounded above or below in a neighborhood of the origin, then precisely one of the following five alternatives holds.
\begin{enumerate}
\item the singularity   is removable, that is, $u$ can be defined at the origin so that $u\in C(B_1)$ and $F(D^2u) = 0$ in $B_1$; \label{enum:pos}
\item $\anom(F) \geq 0$, and $u\underset{0}{\approx}\Phi$;
\item $\anom(\tilde F) \geq 0$, and $u\underset{0}{\approx}-\tilde\Phi$;
\item $\anom(F) < 0$, $u$ can be defined at the origin so that  $ (u(x) - u(0))\underset{0}{\sim} \Phi(x)$; \label{enum:neg-good}
\item $\anom(\tilde F) < 0$, $u$ can be defined at the origin so that $(u(x) - u(0))\underset{0}{\sim}-\tilde \Phi(x)$.\label{enum:neg-bad}
\end{enumerate}
\end{thm}

This theorem generalizes a result of Labutin \cite{Labutin:2001}, who proved (i)-(iii) above under the supplementary assumptions that $F$ is rotationally invariant and there exist (fundamental) solutions $u$ and $v$ of $F(D^2 u) = 0$ in $\R^n \setminus \{ 0 \}$ such that $u(x) \to \infty$ and $v(x) \to -\infty$ as $|x| \to 0$. In light of our results, this latter assumption is equivalent to $\anom(F) \geq 0$ and $\anom(\tilde F)\ge 0$. Alternatives \ENUM{neg-good} and \ENUM{neg-bad} in \THM{singularities} are new even for the Pucci extremal operators.

Our next result is an analogue of \THM{singularities} for solutions of $F(D^2u) = 0$ in $\R^n \setminus B_1$ near infinity. Since we do not have a Kelvin transform available, this is not simply a corollary of \THM{singularities}, although the arguments are very similar.

\begin{thm}\label{thm:singularity-infinity}
Suppose $u \in C(\R^n \setminus B_1)$ is a viscosity solution of the equation
\begin{equation}
F(D^2u) = 0 \quad \mbox{in} \ \R^n \setminus \bar B_1.
\end{equation}
If $u$ is bounded above or below in $\R^n \setminus B_1$, then precisely one of the following five alternatives holds.
\begin{enumerate}
\item  $u_\infty = \lim_{|x| \to \infty} u(x)$ exists, and $\min_{\partial B_r} u \leq u_\infty \leq \max_{\partial B_r} u$, for all $r>1$;
\item $\anom(F) > 0$, $u_\infty := \lim_{|x| \to \infty} u(x)$ exists, and $ (u(x)-u_\infty) \underset{\infty}{\sim} \Phi(x)$; \label{enum:zero}
\item $\anom(\tilde F) > 0$, $u_\infty: = \lim_{|x| \to \infty} u(x)$ exists, and $(u(x)-u_\infty)\underset{\infty}{\sim} - \tilde\Phi(x)$;
\item $\anom(F) \le 0$, and $u\underset{\infty}{\approx}\Phi$;
\item $\anom(\tilde F) \le 0$, and $u\underset{\infty}{\approx}-\tilde\Phi$.
\end{enumerate}
\end{thm}

\medskip

We expect the scaling exponents $\anom(F)$ and $\anom(\tilde F)$ to govern many properties of equations which involve the operator $F$. In addition to the behavior of the fundamental solutions and isolated singularities of $F(D^2u) = 0$, we have described another such property in \REM{anom}, and we do not doubt many others are to come. For instance, the results and techniques in this paper can be used to generalize and sharpen several theorems concerning the removability of singularities, critical exponents and Liouville type results for equations like $F(D^2u)\pm u^p=0$. We refer in particular to  results obtained by Cutr{\`{\i}} and Leoni \cite[Theorems 3.2 and 4.1]{Cutri:2000}, Labutin \cite[Theorem 1]{Labutin:2000}, Felmer and Quaas \cite[Theorem 1.3 and 1.4]{Felmer:2009}.

\medskip

This paper is organized as follows. In the next section, we give some preliminary definitions and recall some standard results for fully nonlinear equations which we use in our arguments. In \SEC{fundysol} we study the scaling number $\anom(F)$ and construct fundamental solutions of \EQ{princ}, establishing the existence part of \THM{fundysol}. The uniqueness part of \THM{fundysol} is a consequence of \THM{liouville}, and is postponed to the end of \SEC{singularities}. In \SEC{examples} we discuss some examples. We study the singularities of solutions of \EQ{princ} and prove our main results in \SEC{singularities}. Finally, in \SEC{stoch} we show that the scaling exponent $\anom(F)$ is related to the behavior of a certain controlled stochastic process.

\section{Preliminaries}

We begin by introducing some notation. The set of $n$-by-$n$ real symmetric matrices is denoted by $\Sy$, and $\iden$ is the identity matrix. If $M,N \in \Sy$, then we write $M \geq N$ if $M-N$ is nonnegative definite. If $x,y \in \R^n$, we denote by $x\otimes y$ the symmetric matrix with entries $\frac{1}{2} ( x_iy_j + x_jy_i)$. If $U$ is a matrix, then the transpose of $U$ is written $U^t$.

For $0 < \lambda \leq \Lambda$ we define the operators
\begin{equation*}
\PucciSub(M) := \sup_{A\in \llbracket\elp,\Elp\rrbracket} \left[ - \trace(AM) \right] \quad \mbox{and} \quad \pucciSub (M) := \inf_{A\in \llbracket\elp,\Elp\rrbracket} \left[ - \trace(AM) \right],
\end{equation*}
for $M\in\Sy$, and where $\llbracket\lambda,\Lambda\rrbracket \subseteq \Sy$ is the subset of $\Sy$ consisting of $A$ for which $\lambda \iden \leq A \leq \Lambda \iden$. The nonlinear operators $\PucciSub$ and $\pucciSub$ are called the \emph{Pucci maximal} and \emph{minimal operators}, respectively. For ease of notation, we will often drop the subscripts and simply write $\Pucci$ and $\pucci$. A convenient way to write the Pucci extremal operators is
\begin{equation}\label{eq:pucci-nice-form}
\Pucci(M) = -\lambda \sum_{\mu_j > 0} \mu_j - \Lambda \sum_{\mu_j < 0} \mu_j \quad \mbox{and} \quad \pucci(M) = -\Lambda \sum_{\mu_j > 0} \mu_j - \lambda \sum_{\mu_j < 0} \mu_j,
\end{equation}
where $\mu_1, \ldots, \mu_n$ are the eigenvalues of $M$.

\medskip

In this article, we require our nonlinear operator $F:\Sy \rightarrow \R$ to be uniformly elliptic in the sense that
\begin{itemize}
\item[(H1)] there exist $0 < \lambda \leq \Lambda$ such that for every $M,N \in \Sy$, \[ \puccisub{\lambda}{\Lambda}(M-N) \leq F(M) - F(N) \leq \Puccisub{\lambda}{\Lambda}(M-N),\]
\end{itemize}
and positively homogeneous of order one:
\begin{itemize}
\item[(H2)] For all $M\in \Sy$ and $t \geq 0$, $F(t M) = t F(M)$.
\end{itemize}
Notice that we have written (H1) in a different but equivalent way from how it appeared in the introduction. We remark that (H1) and (H2) are satisfied for both $F=\pucci$ and $F=\Pucci$, and that these hypotheses imply $\pucci(M) \leq F(M) \leq \Pucci(M)$.

\medskip

Every differential equation and differential inequality in this paper is assumed to be satisfied in the viscosity sense, which is the appropriate notion of weak solution for elliptic equations in nondivergence form. For basic definitions as well as a nice introduction to the theory of viscosity solutions of elliptic equations, we refer to \cite{UsersGuide} and \cite{Caffarelli:Book}. The survey \cite{UsersGuide} is a complete and deep account of the early theory of viscosity solutions, while the book \cite{Caffarelli:Book} describes the more recent regularity theory, made possible by the breakthrough article \cite{Caffarelli:paper}.

To simplify the reader's task, we mention some standard results  which will be used in this article. In what follows we suppose that $\Omega$ is an open subset of $\R^n$, the operator $F$ satisfies (H1), $f\in C(\Omega)$, and $u$ satisfies the differential inequalities
\begin{equation*}
\pucci(D^2u)  \leq |f| \qquad \mbox{and} \qquad \Pucci(D^2u)  \geq -|f| \quad \mbox{in} \ \ \Omega.
\end{equation*}
\begin{itemize}
\item \emph{Strong maximum principle} (\cite[Proposition 4.9]{Caffarelli:Book}). Suppose that $u,v\in C(\bar \Omega)$ satisfy $F(D^2w) \leq f \leq F(D^2v)$ in $\Omega$ and $u \leq v$ in $\Omega$. If $u(x_0) = v(x_0)$ at some point $x_0\in \Omega$, then $u\equiv v$ in $\Omega$.

\item \emph{Harnack inequality} (\cite[Theorem 4.3]{Caffarelli:Book}). Suppose in addition  $u \geq 0$. Then for each compact subsets $K_1\subset K_2 \subseteq \Omega$, there is a constant $C$ depending on $n,\Elp,\elp,K_1,K_2,\Omega,$ such that
\begin{equation*}
\sup_{K_1} u \leq C \left( \inf_{K_1} u + \| f \|_{L^n(K_2)} \right).
\end{equation*}

\item \emph{Local $C^{1,\gamma}$ estimates} (\cite[Theorem 8.3]{Caffarelli:Book}). For each compact subsets $K_1\subset K_2 \subseteq \Omega$, and each $p> n$, there exist constants $0 < \gamma < 1$ and $C$ depending on  $n,p, \Elp,\elp,K_1,K_2,\Omega,$  such that
\begin{equation*}
\| u \|_{C^{1,\gamma}(K_1)} \leq C\left( \| v \|_{L^\infty(K_2)} + \| f \|_{L^p(K_2)} \right).
\end{equation*}

\item \emph{Stability under uniform convergence} (\cite[Proposition 2.9]{Caffarelli:Book}). Suppose that $u_k,f_k \in C(\Omega)$ are such that $F(D^2u_k) \leq f_k$ in $\Omega$ for each $k\geq 1$. Also assume that $u_k \to u$ and $f_k \to f$ locally uniformly in $\Omega$. Then $u$ satisfies the inequality $F(D^2u) \leq f$ in $\Omega$.

\item \emph{Transitivity of inequalities in the viscosity sense} (\cite[Theorem 5.3]{Caffarelli:Book}). Suppose that $G,H$ are nonlinear operators satisfying (H1) such that $F(M)+G(N) \geq H(M+N)$. Suppose also that $F(D^2u) \leq f$ and $v,g\in C(\Omega)$ are such that $G(D^2v) \leq g$. Then the function $w:=u+v$ satisfies $H(D^2w) \leq f+g$.

\item \emph{The supremum of a family of subsolutions is a subsolution} (\cite[Proposition 2.7]{Caffarelli:Book}). Likewise, the infimum of a family of supersolutions is a supersolution.
\end{itemize}

For the rest of this article, we assume that the dimension $n$ of our space is at least 2. For  each $\alpha \in \R$ we define the radial function $\radfun\alpha \in C^\infty(\Rnpunct)$ by \EQ{radfun}. Notice that we have chosen the signs in the definition of $\radfun{\alpha}$ to ensure continuity in the following sense
\begin{equation}
\frac{\xi_\alpha-1}{\alpha}\underset{\alpha\searrow0}{\longrightarrow}\xi_0,\qquad \frac{\xi_\alpha+1}{-\alpha}\underset{\alpha\nearrow0}{\longrightarrow}\xi_0,
\end{equation}
meant to exploit the fact that if $u$ is a solution of \EQ{princ} then so is $au+ b$, for each $a>0$, $b\in \R$.

For each $\alpha \in \R$ and all $\sigma > 0$, we define the rescaling operator $\rescale{\alpha}{\sigma} : C(\Rnpunct) \to C(\Rnpunct)$ by
\begin{equation}
(\rescale\alpha\sigma u)(x) := \left\{ \begin{array}{ll}
\sigma^\alpha u(\sigma x) & \mathrm{if} \ \ \alpha \neq 0, \\
u(\sigma x) + \log(\sigma) & \mathrm{if} \ \ \alpha = 0.
\end{array}\right.
\end{equation}
Notice that the function $\radfun\alpha$ is invariant under the rescaling operator $\rescale{\alpha}{\sigma}$, that is,
$\rescale{\alpha}{\sigma} \radfun{\alpha} = \radfun{\alpha}$ for every $\sigma > 0$. For each $\alpha \in [-1,\infty) \setminus \{ 0 \}$, we define the following spaces of homogeneous functions
\begin{gather*}
\homcls\alpha := \{ v \in C(\R^n \setminus \{ 0 \}) : \alpha v \geq 0, \ \rescale\alpha\sigma v = v \ \mbox{for every} \ \sigma > 0 \},\\
\homclsplus\alpha := \{ v \in C(\R^n \setminus \{ 0 \}) : \alpha v > 0, \ \rescale\alpha\sigma v = v \ \mbox{for every} \ \sigma > 0 \},\\
\intertext{and for $\alpha=0$ we set}
\homcls{0} := \homclsplus{0} := \{ v \in C(\R^n \setminus \{ 0 \}) : \ \rescale 0 \sigma v = v \ \mbox{for every} \ \sigma > 0 \}.
\end{gather*}
We define a special constant $\anom = \anom(F)$ by
\begin{multline}\label{eq:def-anom}
\anom(F) := \sup \left\{ \alpha \in (-1,\infty) \setminus \{ 0 \} : \ \mbox{there exists} \ v\in \homclsplus{\alpha}\ \right.\\
\left. \mbox{such that} \ F(D^2v) \geq 0 \ \mbox{in} \ \R^n \setminus \{ 0 \} \right\}.
\end{multline}
We call $\anom(F)$ the \emph{scaling exponent} of $F$. In order to estimate $\anom(F)$, let us calculate
\begin{equation*}
D^2\radfun\alpha =
\begin{cases}
|\alpha| (\alpha+2) |x|^{-\alpha-4} x\otimes x - |\alpha| |x|^{-\alpha-2}\iden & \mathrm{if} \ \ \alpha \neq 0, \\
2|x|^{-4} x \otimes x - |x|^{-2} \iden & \mathrm{if} \ \ \alpha = 0.
\end{cases}
\end{equation*}
Observe that for $\alpha \neq 0$, the eigenvalues of $D^2\radfun{\alpha}(x)$ are $|\alpha|(\alpha+1) |x|^{-\alpha-2}$ with multiplicity one, and $-|\alpha| |x|^{-\alpha-2}$ with multiplicity $n-1$. Similarly, the eigenvalues of $D^2\radfun{0}$ are $|x|^{-2}$ with multiplicity one and $-|x|^{-2}$ with multiplicity $n-1$. Thus inserting $D^2\radfun{\alpha}(x)$ into the Pucci extremal operators, we discover that
\begin{equation}\label{eq:pucci-radfun}
\pucci \left( D^2\radfun\alpha \right) = \begin{cases}
|\alpha| |x|^{-\alpha-2} \left( \lambda(n-1) - \Lambda(\alpha+1) \right) & \mathrm{if} \ \ \alpha \neq 0, \\
|x|^{-2} \left( \lambda(n-1) - \Lambda \right) & \mathrm{if} \ \ \alpha =0,
\end{cases}
\end{equation}
and
\begin{equation}\label{eq:Pucci-radfun}
\Pucci \left( D^2\radfun\alpha \right) = \begin{cases}
|\alpha| |x|^{-\alpha-2} \left( \Lambda(n-1) - \lambda(\alpha+1) \right) & \mathrm{if} \ \ \alpha \neq 0, \\
|x|^{-2} \left( \Lambda(n-1) - \lambda \right) & \mathrm{if} \ \ \alpha =0.
\end{cases}
\end{equation}
In particular, we see that
\begin{equation} \label{eq:pucci-fundysol}
\mathrm{sign}\left(\pucci(D^2 \radfun{\alpha})\right) = \mathrm{sign} \left(\frac{\lambda}{\Lambda}(n-1) -1-\alpha\right),
\end{equation}
\begin{equation} \label{eq:Pucci-fundysol}
\mathrm{sign}\left(\Pucci(D^2 \radfun{\alpha})\right) = \mathrm{sign} \left(\frac{\Lambda}{\lambda}(n-1) -1-\alpha\right).
\end{equation}
Since $F\geq \pucci$, it immediately follows from \EQ{pucci-fundysol} and the definition of $\anom(F)$ that
\begin{equation*}
\anom(F) \geq \frac{\lambda}{\Lambda}(n-1) -1 > -1.
\end{equation*}
We postpone demonstrating an upper bound for $\anom(F)$, since for this we need a result in the next section (see \COR{anom-bounded}).


\section{Existence of fundamental solutions}\label{sec:fundysol}

In this section we construct fundamental solutions of the equation
\begin{equation}\label{eq:equation}
F(D^2u)=0\quad\mbox{in }\;\R^n\setminus\{0\}.
\end{equation}
More precisely, we prove  the following result, which represents the existence portion of \THM{fundysol}.

\begin{prop}\label{prop:fundysol}
There exists a solution $\Phi \in \homclsplus{\anom(F)}$ of the equation
\begin{equation*}
F(D^2\Phi) = 0 \quad \mbox{in} \ \R^n\setminus \{ 0 \}.
\end{equation*}
The function $\Phi$ is unique in the following sense: If $\alpha > -1$ and $u\in \homclsplus{\alpha}$ is a solution of \EQ{equation}, then $\alpha = \anom(F)$ and either $u \equiv t\Phi$ for some $t>0$, or $u \equiv \Phi+c$ for some $c\in\R$.
\end{prop}

To help the reader avoid misunderstandings, we recall that in this paper the spaces $\homclsplus{\alpha}$ contain positive functions if $\alpha>0$, and negative functions if $\alpha<0$.

We begin by proving a version of the homogeneous comparison principle, suitable for our purposes.

\begin{prop}\label{prop:HCP}
Suppose that $\alpha \geq -1$ and $f \in \homcls{\alpha+2}$, $u\in \homcls{\alpha}$, and $v\in \homclsplus{\alpha}$ satisfy the differential inequalities
\begin{equation}
F(D^2u) \leq f \leq F(D^2v) \quad \mbox{in} \ \R^n\setminus \{ 0 \}.
\end{equation}
Then
\begin{enumerate}
\item if $\alpha > 0$, then either $u \leq v$ or there exists $t > 1$ such that $u \equiv tv$;
\item if $\alpha = 0$, then $u-v$ is constant;
\item if $-1 < \alpha < 0$ and $f\equiv 0$, then either $u \geq v$ or there exists $t > 1$ such that $u \equiv tv$.
\end{enumerate}
\end{prop}
\begin{proof}
First we consider the case $\alpha > 0$. For each $s > 1$, define the function $w_s := u - sv$. Heuristically, using (H1) we have
\begin{equation}\label{eq:HCP-ws}
\pucci(D^2w_s) \leq F(D^2 u) - sF(D^2 v) \leq f - s f \leq 0 \quad \mbox{in} \ \R^n \setminus \{ 0 \},
\end{equation}
for every $s> 1$. Using the transitivity of differential inequalities in the viscosity sense, we see that $w_s$ is a viscosity subsolution of $\pucci(D^2 w_s) \leq 0$. Define
\begin{equation}\label{eq:HCP-def-t}
t := \inf \left\{ s > 1 : w_s < 0 \ \ \mbox{in} \ \R^n \setminus \{ 0 \} \right\}.
\end{equation}
Our hypotheses imply that $1 \leq t < \infty$, and $w_t \leq 0$. If $t=1$, then we conclude that $u \leq v$, and we have nothing left to show.

We have left to examine the case $1 < t < \infty$. We must show that $w_t \equiv 0$. By the strong maximum principle, it is enough to show that $w_t(x_0) = 0$ for some $x_0 \in \R^n \setminus \{ 0 \}$. If not, we have $-w_t > \delta v$ on $\partial B_1$ (and hence on $\R^n\setminus \{ 0 \}$, by the homogeneity of $w_t$ and $v$),  for some $0< \delta < t-1$. It follows that $w_{t-\delta} < 0$, a contradiction to \EQ{HCP-def-t}. This completes the proof of (i).

Suppose now that $\alpha = 0$. Define the function $w : = u - v$, which is constant on the set $\{ tx : t> 0\}$ for each $x \in \R^n \setminus \{ 0 \}$. Moreover,
\begin{equation*}
\pucci(D^2 w) \leq 0 \quad \mbox{in} \ \R^n \setminus \{ 0 \}.
\end{equation*}
Set
\begin{equation*}
M := \max_{\partial B_1} w= \sup_{\R^n \setminus \{ 0 \}} w.
\end{equation*}
By the strong maximum principle, $w \equiv M$. This verifies (ii).

Finally, let us prove (iii). If $u \not \equiv 0$, then by the strong maximum principle $u \in \homclsplus{\alpha}$. Let $w_s$ be defined as above, and notice that for $s> 0$ small enough we have $w_s < 0$ in $\R^n \setminus \{ 0 \}$. Moreover, we have $\pucci(D^2w_s) \leq 0$ in $\R^n \setminus \{ 0 \}$ for all $s>0$. Set
\begin{equation*}
t := \sup \left\{ s > 0 : w_s < 0 \ \ \mbox{in} \ \R^n \setminus \{ 0 \} \right\}.
\end{equation*}
Then $0 = \sup w_t = \max_{\partial B_1} w_t$, and by the strong maximum principle we conclude that $w_t \equiv 0$. We have shown that either $u \equiv 0$ or $u \equiv tv$ for some $t > 0$, from which the result follows.
\end{proof}

The next lemma establishes that the set of $\alpha > -1$ for which there exists a supersolution $u\in \homclsplus{\alpha}$ of $F(D^2u) \geq 0$ in $\R^n \setminus \{ 0 \}$ is an interval.

\begin{lem}\label{lem:strict-supersolution}
Assume that $-1 < \alpha < \anom(F)$. Then there exists a supersolution $u \in \homclsplus{\alpha}$ of the inequality
\begin{equation}\label{eq:strict-supersolution}
F(D^2u ) \geq \radfun{\alpha+2} \quad \mbox{in} \ \R^n\setminus \{ 0 \}.
\end{equation}
\end{lem}
\begin{proof}
Select $\beta > -1$ satisfying $\alpha < \beta < \anom$ and for which there exists a supersolution $v \in \homclsplus{\beta}$ of the inequality
\begin{equation}\label{eq:strict-supersolution-beta}
F(D^2v) \geq 0 \quad \mbox{in} \ \R^n\setminus \{ 0 \}.
\end{equation}

First we suppose that $0 < \alpha < \beta$. Define $\tau := \beta / \alpha > 1$ and $w(x) : = \left( v(x) \right)^{1/\tau}$. Notice that $w \in \homclsplus{\alpha}$. Formally, we have
\begin{equation*}
F(D^2w) = \frac{1}{\tau} v^{1/\tau-1} F\left(D^2v - (1- 1/\tau) v^{-1} Dv\otimes Dv \right) \geq \frac{\lambda(\tau-1)|Dv|^2 v^{1/\tau}}{\tau^2v^2}.
\end{equation*}
As $v \in \homclsplus{\beta}$, by differentiating $v(x)=t^\beta v(tx)$ with respect to $t$ we get $x\cdot Dv = -\beta v$, hence $|Dv| \geq \beta |x|^{-1} v$. Thus we formally estimate
\begin{equation}\label{eq:strict-supersolution-wts}
F(D^2w) \geq \tau^{-2} \lambda(\tau -1)\beta^2 v^{1/\tau}|x|^{-2} = \lambda\alpha^2 (\tau-1) w|x|^{-2}\ge c(\min_{\partial B_1}w)\xi_{\alpha+2}.
\end{equation}
To verify \EQ{strict-supersolution-wts} in the viscosity sense, we select a smooth test function $\varphi$ and $x_0 \in \R^n \setminus \{ 0 \}$ such that$
x\mapsto w(x) - \varphi(x)$ has a local minimum at $x = x_0$.
We must demonstrate that
\begin{equation}\label{eq:strict-supersolution-wts-2}
F(D^2\varphi(x_0)) \geq \lambda \alpha^2 (\tau-1) w(x_0).
\end{equation}
We may suppose without loss of generality that $\varphi(x_0) = w(x_0)$ and $\varphi > 0$. Let $\psi(x):= \left( \varphi(x)\right)^{\tau}$. Then $v(x_0) = \varphi(x_0)$ and $
x\mapsto v(x) - \psi(x)$ has a local minimum at $ x = x_0$.
Recalling \EQ{strict-supersolution-beta}, we have
$
F(D^2\psi(x_0)) \geq 0.
$
A routine calculation reveals that
\begin{equation*}
D^2\psi(x) = \tau \left( \varphi(x)\right)^{\tau-1} D^2\varphi(x) + \tau \left( \tau - 1 \right) \left( \varphi(x)\right)^{\tau-2} D\varphi(x)\otimes D\varphi(x).
\end{equation*}
Thus
\begin{align*}
0 & \leq F\left( D^2\varphi(x_0) + \left( \tau - 1 \right) \left( \varphi(x_0)\right)^{-1} D\varphi(x_0)\otimes D\varphi(x_0) \right) \\
& \leq F\left( D^2\varphi(x_0) \right) + \frac{\tau -1}{\varphi(x_0)} \Pucci \left( D\varphi(x_0)\otimes D\varphi(x_0) \right) \\
& = F \left(D^2\varphi(x_0) \right) - \frac{\elp (\tau -1) }{w(x_0)} |D\varphi(x_0)|^2.
\end{align*}
Rearrange to write
\begin{equation}\label{eq:strict-supersolution-1}
F\left(D^2\varphi(x_0) \right) \geq \frac{\lambda(\tau -1)}{w(x_0)} |D\varphi(x_0)|^2.
\end{equation}
We will next derive a lower bound for $|D\varphi(x_0)|$. Owing to the homogeneity of $w$, at any point $x\neq 0$ we have
\begin{equation*}
 \frac{\partial}{\partial s} \left. \left( w(x+sx) \right) \right|_{s=0}  = w(x)  \frac{\partial}{\partial s}\left. (1+s)^{-\alpha} \right|_{s=0}
= -\alpha w(x).
\end{equation*}
Since $w-\varphi$ has a maximum at $x_0$, we see that
\begin{equation*}
x_0 \cdot D\varphi(x_0)  = \frac{\partial}{\partial s} \left. \left( \varphi(x_0 + sx_0) \right) \right|_{s=0} \leq -\alpha w(x_0).
\end{equation*}
Hence
\begin{equation*}
|D\varphi(x_0)| \geq \frac{\alpha w(x_0)}{|x_0|}.
\end{equation*}
Inserting into \EQ{strict-supersolution-1}, we obtain \EQ{strict-supersolution-wts-2}. Recalling that $w \in \homclsplus{\alpha}$, we see that a large multiple of $w$ satisfies \EQ{strict-supersolution}.

In the case $\alpha = 0<\beta$, we define $w(x) : = \beta^{-1} \log v(x)$. Then $w \in \homclsplus{0}$, and formally we see that
\begin{equation*}
F(D^2w) = F\left( \frac{D^2v}{\beta v} - \frac{Dv\otimes Dv}{\beta v^2} \right) \geq \frac{\lambda |Dv|^2}{\beta v^2} \geq \lambda \beta |x|^{-2} \quad \mbox{in} \ \R^n \setminus \{ 0 \}.
\end{equation*}
This differential inequality is easily verified in the viscosity sense, as we argued above in the proof of \EQ{strict-supersolution-wts}.

Similarly, in the case $\alpha<0=\beta$, we define $w(x):= -\exp(\alpha v(x))$. It is easily verified that $w \in \homclsplus{\alpha}$, and formally we have
\begin{equation*}
F(D^2w) = F( w D^2v - \alpha w Dv\otimes Dv) \geq - \lambda \alpha |w| |Dv|^2 \geq \lambda |\alpha| |w| |x|^{-2} \quad \mbox{in} \ \R^n \setminus \{ 0 \}.
\end{equation*}
This inequality can also be routinely verified in the viscosity sense, so that some positive multiple $u$ of $w$ satisfies \EQ{strict-supersolution}. Likewise, if $\alpha<0<\beta$ we can combine the last two cases to obtain the desired supersolution.

Finally, we consider the case that $-1 < \alpha < \beta < 0$. With $\tau := \beta / \alpha < 1$, we define $w(x) := - (- v(x) )^{1/\tau}$. Formally we compute
\begin{align*}
F(D^2w) & = \tau^{-1} (-v)^{1/\tau -1} F\left( D^2v - (1/\tau -1) (-v)^{-1} Dv\otimes Dv \right)  \\
& \geq \frac{\lambda(1-\tau) (-v)^{1/\tau}|Dv|^2}{\tau^2 (-v)^2}  \\
& \geq \lambda(\tau -1)  |x|^{-2} \beta w.
\end{align*}
Since $\lambda(\tau -1)  |x|^{-2} \beta w \geq c |x|^{-\alpha -2}$ in $\R^n\setminus \{  0  \}$, we may again argue as above to conclude that a multiple of $w$ satisfies \EQ{strict-supersolution}.
\end{proof}

From the previous two results we deduce a maximum principle.

\begin{cor}\label{cor:alpha-minus-maximum-principle}
Assume that $-1 \leq \alpha < \anom(F)$, $\alpha \neq 0$. Suppose that $u \in \homcls{\alpha}$ satisfies the inequality
\begin{equation}\label{eq:alpha-mmp}
F(D^2u ) \leq 0 \quad \mbox{in} \ \R^n \setminus \{ 0 \}.
\end{equation}
Then $u \equiv 0$. If $\anom(F) > 0$, then there does not exist a function $u \in \homclsplus{0}$ satisfying the inequality \EQ{alpha-mmp}.
\end{cor}
\begin{proof}
According to \LEM{strict-supersolution}, there exists a function $v \in \homclsplus{\alpha}$ which satisfies
\begin{equation*}
F(D^2v) \geq \radfun{\alpha+2} \quad \mbox{in} \ \R^n \setminus \{ 0 \}.
\end{equation*}
If $\alpha \neq 0$, then according to \PROP{HCP}, either $|u| \leq c |v|$ for every $c > 0$, or $u \equiv tv$ for some $t > 0$. The first alternative implies that $u \equiv 0$. The second alternative is not possible since $u$ is a subsolution and $v$ is a strict supersolution of $F(D^2u) = 0$. If $\alpha = 0$, then we deduce that $u-v$ is constant, which is impossible.
\end{proof}

\begin{cor}\label{cor:Puccis-fundysols}
For any $0 < \elp \leq \Elp$,
\begin{equation}\label{eq:Puccis-fundysols}
\anom\left( \pucciSub \right) = \frac{\lambda}{\Lambda}(n-1) -1 \quad \mbox{and} \quad \anom\left( \PucciSub \right) = \frac{\Lambda}{\lambda}(n-1) -1.
\end{equation}
\end{cor}
\begin{proof}
Recalling \EQ{pucci-fundysol}, from the definition of $\anom$ we see that
\begin{equation*}
\anom(\pucci) \geq \frac{\lambda}{\Lambda}(n-1) -1.
\end{equation*}
However, if $\anom(\pucci) > \frac{\lambda}{\Lambda}(n-1) -1$, then we see that \COR{alpha-minus-maximum-principle} and \EQ{pucci-fundysol} are incompatible. This verifies the first equality in \EQ{Puccis-fundysols}. The second equality is proved with a similar argument.
\end{proof}

\begin{cor}\label{cor:anom-bounded}
For any operator $F$ which satisfies (H1) and (H2),
\begin{equation*}
 \frac{\lambda}{\Lambda}(n-1) -1 \leq \anom(F) \leq \frac{\Lambda}{\lambda}(n-1) -1.
\end{equation*}
\end{cor}
\begin{proof}
Since  $\pucci \leq F \leq \Pucci$, the result immediately follows from \EQ{Puccis-fundysols} and the definition of $\anom(F)$.
\end{proof}

We now split the proof of  \PROP{fundysol} into two parts, and consider separately the cases $\anom(F)>0$ and $\anom(F)\le 0$.

\begin{lem}\label{lem:increment-alpha}
Suppose $\alpha \geq 0$, $f \in H_{\alpha+2}^+$, and $u\in \homclsplus{\alpha}$ satisfy
\begin{equation}
F(D^2 u) = f \quad \mbox{in} \ \R^n \setminus \{ 0 \}.
\end{equation}
Then $\alpha < \anom(F)$.
\end{lem}
\begin{proof}
Employing the local $C^{1,\gamma}$ estimates for uniformly elliptic equations, we deduce that $u \in C^1(\R^n\setminus \{ 0 \})$. Set $k: = \sup_{\partial B_1} | Du |$. By the homogeneity of $u$, we have
\begin{equation}\label{eq:increment-alpha-Du}
|Du(x)| \leq k |x|^{-\alpha -1} \quad \mbox{for every} \ \ x\in \R^n \setminus \{ 0 \}.
\end{equation}
First we consider the case $\alpha > 0$. Let $\frac{1}{2} < \tau < 1$ be a number to be selected, and set
\begin{equation*}
w(x) : = \left( u(x) \right)^{1/\tau}.
\end{equation*}
Notice that $w \in \homclsplus{\beta}$ for $\beta := \alpha / \tau > \alpha$. From \EQ{increment-alpha-Du} we easily obtain the estimate
\begin{equation}\label{eq:increment-alpha-Dw}
|Dw(x)| \leq C |x|^{-\beta -1}.
\end{equation}
We claim that if $\tau$ is selected sufficiently close to $1$, then $w$ satisfies the inequality
\begin{equation} \label{eq:increment-alpha-1}
F(D^2w) \geq 0 \quad \mbox{in} \ \R^n \setminus \{ 0 \}.
\end{equation}
Take a smooth test function $\varphi$ and a point $x_0 \neq 0$ such that $\varphi(x_0) = w(x_0)$, and
$x \mapsto w(x) - \varphi(x)$ has a local minimum at $ x = x_0$.
Observe that $Dw(x_0) = D\varphi(x_0)$. Set $\psi := \varphi^\tau$. The function
$x \mapsto u(x) - \psi(x)$ has a local minimum at $x = x_0$.
Thus
\begin{equation*}
F \left(D^2 \psi(x_0) \right) \geq f(x_0).
\end{equation*}
Following the calculations in the proof of \LEM{strict-supersolution}, we obtain the estimate
\begin{equation*}
\frac{f(x_0)}{\tau (w(x_0))^{\tau-1}} \leq F\left(D^2\varphi(x_0)\right) + \frac{\Lambda(1-\tau)}{\varphi(x_0)}|D\varphi(x_0)|^2.
\end{equation*}
Rearranging, we have
\begin{align*}
F \left( D^2\varphi(x_0) \right) & \geq c|x_0|^{-\alpha -2} (w(x_0))^{1-\tau} - \frac{C (1-\tau)}{w(x_0)}|x_0|^{-2\beta-2} \\
& \geq c |x_0|^{-\beta -2}  - C(1-\tau) |x_0|^{-\beta -2}.
\end{align*}
Taking $1-\tau > 0$ to be sufficiently small, we obtain $
F\left( D^2\varphi(x_0) \right) \geq 0$,
which verifies that for such $\tau$ the function $w$ satisfies \EQ{increment-alpha-1}. It now follows from the definition \EQ{def-anom} of $\anom(F)$ that $\alpha < \beta \leq \anom(F)$.

Next we consider the case $\alpha = 0$. Define the function $v: = \exp(\beta u)$, where $\beta > 0$ will be selected. Then $v \in \homclsplus{\beta} \cap C^1(\R^n \setminus \{ 0 \})$, and if $u \in C^2$ we check that
\begin{equation*}
D^2u = \frac{1}{\beta} \frac{D^2v}{v} - \frac{1}{\beta} \frac{Dv \otimes Dv}{v^2}.
\end{equation*}
Formally, for some $c> 0$ we have
\begin{equation*}
c \beta v |x|^{-2} \leq F\left(D^2 v - \frac{1}{v} Dv\otimes Dv\right) \leq F(D^2 v) + \frac{\Lambda}{v} |Dv|^2.
\end{equation*}
This calculation can be made rigorous by arguing with smooth test functions, so that in the viscosity sense we have
\begin{equation*}
F(D^2v) \geq c \beta v |x|^{-2} -  \frac{\Lambda}{v} |Dv|^2.
\end{equation*}
Using $|Dv| = \beta |Du| v$ and the estimate \EQ{increment-alpha-Du}, we obtain
\begin{equation*}
F(D^2v) \geq c \beta v |x|^{-2} - \Lambda k^2 \beta^2 v |x|^{-2}.
\end{equation*}
Thus $F(D^2v) \geq 0$ in $\R^n \setminus \{ 0 \}$, provided that we select $\beta : = c / \Lambda k^2 > 0$.
\end{proof}

The next lemma is the key to the proof of \PROP{fundysol} in the case $\anom(F) > 0$.

\begin{lem}\label{lem:existence-alpha}
Suppose that $0 < \alpha < \anom(F)$ and $f\in \homcls{\alpha+2}$. Then there exists a unique solution $u \in \homcls{\alpha}$ of the equation
\begin{equation}\label{eq:existence-alpha}
F(D^2u) = f \quad \mbox{in} \ \R^n \setminus \{ 0 \}.
\end{equation}
Moreover, if $f\not\equiv 0$, then $u \in \homclsplus{\alpha}$.
\end{lem}
\begin{proof}
According to \LEM{strict-supersolution}, there exists a supersolution $w\in  \homclsplus{\alpha}$ of
\begin{equation*}
F(D^2w) \geq f \quad \mbox{in} \ \R^n \setminus \{ 0 \}.
\end{equation*}
Let us define
\begin{equation*}
u(x) := \sup \left\{ \tilde u(x) : \tilde u \in C(\R^n \setminus \{ 0 \}) \ \mbox{is a subsolution of} \ \EQ{existence-alpha}, \ \mbox{and} \ \tilde u \leq w \right\}.
\end{equation*}
Obviously the zero function is a subsolution of \EQ{existence-alpha}, so $u$ is well-defined and $u\geq 0$. If $\tilde u \in C(\R^n \setminus \{ 0 \})$ is a subsolution of \EQ{existence-alpha}, then so is $\rescale{\alpha}{\sigma} \tilde u$ for any $\sigma > 0$, by the scaling invariance of the equation. Thus $u \in \homcls{\alpha}$ by construction. Standard arguments from viscosity solution theory (see \cite{UsersGuide}) imply that $u$ is a solution of \EQ{existence-alpha}. The uniqueness of $u$ follows at once from \PROP{HCP}. If $f\not \equiv 0$, then by the strong maximum principle $u> 0$ and hence $u\in \homclsplus{\alpha}$.
\end{proof}

We are now ready to prove \PROP{fundysol} in the case that $\anom(F) > 0$.

\begin{prop}\label{prop:fundysol-case1}
Suppose that $\anom(F) > 0$. Then there exists a function $\Phi \in \homclsplus{\anom}$ such that
\begin{equation*}
F\left(D^2 \Phi\right) = 0 \quad \mbox{in} \ \R^n\setminus \{ 0 \}.
\end{equation*}
Moreover, if $\beta > -1$ and $u \in \homclsplus{\beta}$ satisfy $F(D^2u) = 0$ in $\R^n \setminus \{ 0 \}$, then $\beta = \anom(F)$ and $u \equiv t \Phi$ for some $t > 0$.
\end{prop}
\begin{proof}
For each $0 < \alpha < \anom$, let $u_\alpha \in \homclsplus{\alpha}$ denote the unique solution of
\begin{equation*}
F(D^2u_\alpha) = \radfun{\alpha+2} \quad \mbox{in} \ \R^n \setminus \{ 0 \}.
\end{equation*}
We claim that
\begin{equation}\label{eq:existence-blowup}
\sup_{|x| =1} u_\alpha(x)  \to + \infty \quad \mbox{as} \ \ \alpha \to \anom.
\end{equation}
Suppose on the contrary that there exists a sequence $\alpha_j \to \alpha^*$ such that
\begin{equation*}
\sup_{j\geq 1} \sup_{|x|=1} u_{\alpha_j} (x) \leq C.
\end{equation*}
By the homogeneity of the functions $u_\alpha$, it follows easily that
\begin{equation*}
\sup_{x \in K} u_{\alpha_j} \leq C
\end{equation*}
for any compact subset $K \subseteq \R^n \setminus \{ 0 \}$. Therefore we have the estimate
\begin{equation*}
\| u_{\alpha_j} \|_{C^\gamma(K)} \leq C.
\end{equation*}
By taking a subsequence, if necessary, we may assume that $u_{\alpha_j}$ converges locally uniformly on $\R^n \setminus \{ 0 \}$ to a function $u \in C(\R^n \setminus \{ 0 \})$. It is immediate that $u \in \homclsplus{\anom}$ and $u$ is a solution of the equation
\begin{equation*}
F(D^2u) = \radfun{\anom+2} \quad \mbox{in} \ \R^n \setminus \{ 0 \}.
\end{equation*}
This contradicts \LEM{increment-alpha} and the definition of $\anom$. Therefore \EQ{existence-blowup} holds.

Define the functions $v_\alpha$ by
\begin{equation*}
v_\alpha(x) := c_\alpha^{-1} u_\alpha(x), \quad \mbox{where} \ \ c_\alpha := \sup_{|x| =1} u_\alpha(x).
\end{equation*}
Then $v_\alpha \in \homclsplus{\alpha}$. In fact, using homogeneity and the Harnack inequality we have
\begin{equation*}
c \radfun{\alpha} \leq v_\alpha \leq \radfun{\alpha} \quad \mbox{in} \ \ \R^n \setminus \{ 0 \}
\end{equation*}
for some $c > 0$. Using the homogeneity of $F$, we see that $v_\alpha$ is a solution of
\begin{equation*}
F(D^2v_\alpha) = c_\alpha^{-1} \radfun{\alpha+2} \quad \mbox{in} \ \ \R^n \setminus \{ 0 \}.
\end{equation*}
For every compact subset $K\subseteq \R^n\setminus \{ 0 \}$, we have the estimate
\begin{equation*}
\| v_\alpha \|_{C^\gamma(K)} \leq C.
\end{equation*}
Thus there exists a function $\Phi \in C(\R^n\setminus \{ 0 \} )$ such that, up to a subsequence,
\begin{equation*}
v_\alpha \rightarrow \Phi \quad \mbox{locally uniformly on} \ \ \R^n \setminus \{ 0 \}.
\end{equation*}
It immediately follows that $\Phi \in \homclsplus{\anom}$ and
\begin{equation*}
c \radfun{\anom} \leq \Phi \leq \radfun{\anom}.
\end{equation*}
The uniqueness assertions in the last statement in the proposition are immediately obtained from \PROP{HCP} and \COR{alpha-minus-maximum-principle}.
\end{proof}

The proof of the existence of the fundamental solution in the case $\anom(F) \leq 0$ is more subtle. Because the inequality in the conclusion of \PROP{HCP} for $\alpha < 0$ is the reverse of the case $\alpha > 0$, we expect supersolutions $u\in \homclsplus{\alpha}$ of $F(D^2u) = 0$ to lie \emph{below} subsolutions. Thus we do not know how to extend \LEM{existence-alpha} to $\alpha < 0$. Instead, our proof of \PROP{fundysol} in the case $\anom(F) \leq 0$ relies on the following version of the Leray-Schauder alternative (c.f. Rabinowitz \cite{Rabinowitz:notes} or Chang \cite{Chang:book}).

\begin{prop}[Leray-Schauder alternative]\label{prop:leray-schauder}
Let $X$ be a real Banach space, $K \subseteq X$ a convex cone, and $\mathcal A: \R \times K \to K$ be a compact and continuous mapping such that $\mathcal A (0, u) = 0$ for every $u \in K$. Then there exist unbounded, connected sets $\mathcal{C}^+ \subseteq [0,\infty) \times K$ and $\mathcal{C}^- \subseteq (-\infty,0] \times K$ such that $(0,0) \in \mathcal C^+ \cap \mathcal C^-$ and
\begin{equation*}
\mathcal A(\lambda, u) = u \quad \mbox{for every} \ (\lambda,u) \in \mathcal{C}^+ \cup \mathcal{C}^-.
\end{equation*}
\end{prop}

We use the Leray-Schauder alternative to control the norms of approximate fundamental solutions. We apply it to the Banach space $X = C(\partial B_1)$, and the convex cone $K := \{ u \in C(\partial B_1) : u \leq 0 \}$. Observe that for each $\alpha < 0$ the convex cone $H_\alpha$ is isomorphic to $K$ via the map $u(x) \to \tilde u(x):=u(x/|x|)$.

The following lemma will provide the map $\mathcal A$ to which we are going to apply \PROP{leray-schauder}.
\begin{lem}
For every $-1\leq \alpha, \beta < 0$ and $v \in H_\alpha$, there exists a unique function $u \in H_\alpha^+$ that satisfies the equation
\begin{equation}
\label{eq:ls-operator-pde}
F(D^2 u + \alpha |x|^{-2} (u - v) I_n) = |x|^{-2}(\beta u - \alpha v) + \alpha |x|^{-\alpha-2} \mrbox{in} \R^n \setminus \{ 0 \}.
\end{equation}
Moreover, we have the estimate
\begin{equation}
\label{eq:ls-operator-estimate}
\max_{\partial B_1} |u| \leq \frac{C_1 \alpha}{\beta} (1 + \max_{\partial B_1} |v|),
\end{equation}
for some constant $C_1=C_1(n,\Lambda) > 0$.
\end{lem}

\begin{proof}
Notice that the zero function is a smooth, strict supersolution of \EQ{ls-operator-pde} since
\begin{equation*}
F\!\left( -\alpha |x|^{-2} v(x) \iden \right) \geq 0 > -\alpha v |x|^{-2} + \alpha |x|^{-\alpha-2} \quad \mbox{for every} \ x \in \R^n \setminus \{ 0 \}.
\end{equation*}
Consider the function $w(x) := - C |x|^{-\alpha}$, where we select $C>0$ below. Inserting $w$ into \EQ{ls-operator-pde}, we discover that
\begin{align*}
F\!\left( D^2 w + \alpha |x|^{-2} (w - v) I_n \right) & = F\!\left( -C\alpha(\alpha+2)|x|^{-\alpha-4}(x\otimes x) - (\alpha v)|x|^{-2} \iden \right) \\
& \leq \Pucci\!\left( -C\alpha(\alpha+2)|x|^{-\alpha-4}(x\otimes x) - (\alpha v)|x|^{-2} \iden \right) \\
& \leq n\Lambda |x|^{-2} (\alpha v).
\end{align*}
Select
\begin{equation*}
C:= \frac{\alpha}{\beta} \left( n\Lambda +1\right) \left( 1+ \max_{\partial B_1} |v|  \right),
\end{equation*}
so that
\begin{align*}
n\Lambda |x|^{-2} (\alpha v) & \leq (-C\beta +\alpha) |x|^{-2-\alpha} -\alpha |x|^{-2} v \\
& = |x|^{-2}(\beta u - \alpha v) + \alpha |x|^{-\alpha-2},
\end{align*}
allowing us to conclude that $w$ is a subsolution of \EQ{ls-operator-pde}.

Let us define the function
\begin{equation*}
u(x) := \sup\left\{ w(x) : w \in C(\R^n \setminus \{ 0 \}) \ \mbox{is a subsolution of} \ \EQ{ls-operator-pde} \ \mbox{and} \ w \leq 0 \right\}.
\end{equation*}
It is clear that $-C|x|^{-\alpha} \leq u(x) \leq 0$ for every $x \in \R^n \setminus \{ 0 \}$, giving us \EQ{ls-operator-estimate} for $C_1:= (n\Lambda +1)$. Moreover, we have $u \in \homcls{\alpha}$ due to the scaling invariance of \EQ{ls-operator-pde} and the definition of $u$ --- since if $w$ is a subsolution, then so is $T^\alpha_\sigma w$, for all $\sigma>0$. Standard viscosity solution arguments (c.f. \cite{UsersGuide}) imply that $u$ is a solution of \EQ{ls-operator-pde}. Since the zero function is a smooth strict supersolution, it cannot touch $u$ from above, as $u$ is viscosity subsolution. Thus $u \in \homclsplus{\alpha}$.

To establish the uniqueness of $u$, we notice that the equation is ``proper" with respect to the space $\homclsplus{\alpha}$. By this we mean that the function $u+c|x|^{-\alpha}$ is a strict supersolution of \EQ{ls-operator-pde} for any $c>0$, a fact which is easy to check. Let us suppose that $\tilde u$ is another solution of \EQ{ls-operator-pde} such that $c: = \max_{\partial B_1} (\tilde u - u) > 0$. Then by the strong maximum principle, we must have $\tilde u \equiv u + c|x|^{-\alpha}$, which is impossible since $u+ c|x|^{-\alpha}$ is a \emph{strict} supersolution. The uniqueness of $u$ follows.
\end{proof}

Using the Leray-Schauder alternative and the solution operator from the previous lemma, we build approximate fundamental solutions.

\begin{lem}\label{lem:approx-fundysols}
For every $k > 1$ and $-1 < \beta <0$, there exists a number $\alpha < 0$ satisfying
\begin{equation} \label{eq:bound-alpha}
\min\{ \anom,\beta \} < \alpha < c \beta,
\end{equation}
for some constant $0 < c = c(n,\Lambda) < 1/2$, and a function $u \in H_\alpha^+$ satisfying the equation \EQ{ls-operator-pde} with $u=v$, that is,
\begin{equation}
\label{eq:ls-fp-pde}
F(D^2 u) = |x|^{-2}(\beta - \alpha) u + \alpha |x|^{-\alpha - 2} \mrbox{in} \R^n \setminus \{ 0 \},
\end{equation}
and for which
\begin{equation}
\label{eq:ls-fp-norm}
\max_{\partial B_1} |u|= k.
\end{equation}

\end{lem}

\begin{proof}
Recall that $K := \{ w \in C(\partial B_1) : u \leq 0 \}$. Given  $(\alpha, w) \in [-1, 0) \times K$, let $u \in H_\alpha^+$ be the unique solution of \EQ{ls-operator-pde} for $v(x) := |x|^{-\alpha} w(x/|x|)$, and define an operator $\mathcal{A} : [-1,0) \times K \to K$ by setting $\mathcal{A}(\alpha, \tilde v) := \tilde u$. Let us extend the domain by setting $\mathcal{A}(\alpha, \tilde v) := 0$ for all $\alpha \geq 0$, and $\mathcal{A}(\alpha, \tilde v) := \mathcal{A}(-1, \tilde v)$ for $\alpha < -1$.

By using \EQ{ls-operator-estimate} as well as H\"older estimates and stability properties of viscosity solutions under local uniform convergence, the map
\begin{equation*}
\mathcal{A} : \R \times K \rightarrow K.
\end{equation*}
is easily seen to be continuous and compact.

We now apply \PROP{leray-schauder} to deduce the existence of an unbounded and connected set $\mathcal{C} \subseteq (-\infty, 0] \times K$ such that $(0,0) \in \mathcal{C}$ and
\begin{equation*}
\mathcal{A}(\alpha,\tilde u) = \tilde u \quad \mbox{for every} \ (\alpha, \tilde u) \in \mathcal{C}.
\end{equation*}
We claim that
\begin{equation} \label{eq:C-trapped}
\mathcal{C} \subseteq \left(\min\{\beta,\anom\} , 0\right] \times K.
\end{equation}
Suppose that $-1 < \alpha < \beta$ and $\tilde u \in K$ such that $(\alpha,\tilde u) \in \mathcal{C}$. Then we see that the function $u(x):= |x|^{-\alpha} \tilde u(x/|x|)$ belongs to $\homclsplus{\alpha}$ and satisfies
\begin{equation*}
F(D^2u) \leq 0 \quad \mbox{in} \ \R^n \setminus \{ 0 \}.
\end{equation*}
By \COR{alpha-minus-maximum-principle}, we see that $\alpha > \anom(F)$. Thus $\mathcal{C} \cap [-1,\min\{ \beta, \anom \}] \times K = \emptyset$. Since $\mathcal{C}$ is connected, we deduce \EQ{C-trapped}.

Since $\mathcal{C}$ is unbounded and connected, we can find $(\alpha,\tilde u) \in \mathcal{C}$ such that $u(x) := |x|^{-\alpha} \tilde u(x/|x|)$ satisfies \EQ{ls-fp-norm}. Since $-1 < \alpha < 0$, it is clear that $u$ also satisfies equation \EQ{ls-fp-pde}. Finally, we notice that \EQ{ls-operator-estimate} and \EQ{ls-fp-norm} give us
\begin{equation*}
\frac{1}{2C_1} < \frac{k}{C_1(1+k)} \leq \frac{\alpha}{\beta},
\end{equation*}
where $C_1= (n\Lambda +1)$ is the constant in \EQ{ls-operator-estimate}. Thus if we select $c:= 1/2C_1$, then the second inequality in \EQ{bound-alpha} must hold.
\end{proof}

We now construct fundamental solutions in the case $\alpha^*(F) < 0$, using the approximate fundamental solutions from \LEM{approx-fundysols}.

\begin{prop}
\label{prop:fundysol-case2}
If $\alpha^*(F) < 0$, then there exists a solution $\Phi \in H_{\alpha^*}^+$ of the equation
\begin{equation*}
F(D^2 \Phi) = 0 \mrbox{in} \R^n \setminus \{ 0 \}.
\end{equation*}
Moreover, if $\beta > -1$ and $u \in H_\beta^+$ satisfies $F(D^2 u) = 0$ in $\R^n \setminus \{ 0 \}$, then $\beta = \alpha^*(F)$ and $u \equiv t \Phi$ for some $t > 0$.
\end{prop}

\begin{proof}
Choose a sequence $1 < k_j \to \infty$, and use \LEM{approx-fundysols} to find numbers $\alpha_j$ such that
\begin{equation*}
\anom < \alpha_j < c\anom < 0,
\end{equation*}
and  $u_j \in H_{\alpha_j}^+$ which satisfy the equation
\begin{equation*}
F(D^2 u_j) = |x|^{-2}(\alpha^* - \alpha_j) u_j + \alpha_j |x|^{-\alpha_j - 2} \mrbox{in} \R^n \setminus \{ 0 \},
\end{equation*}
as well as
\begin{equation*}
\max_{\partial B_1} |u_j| = k_j.
\end{equation*}
By taking a subsequence, we may assume that $\alpha_j \to \alpha'$ as $j\to \infty$ for some number $\anom \leq \alpha' \leq c\anom < 0$.

Define $w_j(x) := u_j(x) / k_j$. Observe that
\begin{equation} \label{eq:est-wj}
\max_{\partial B_1} |w_j| = 1,
\end{equation}
and that $w_j$ is a solution of the equation
\begin{equation*}
F(D^2 w_j) = |x|^{-2} (\alpha^* - \alpha_j) w_j + \alpha k_j^{-1} |x|^{-\alpha_j - 2} \mrbox{in} \R^n \setminus \{ 0 \}.
\end{equation*}
Since the right-hand side of the expression above is locally uniformly bounded, H\"older estimates imply that for any compact $K \subseteq \R^n \setminus \{ 0 \}$,
\begin{equation*}
\| w_j \|_{C^\gamma(K)} \leq C,
\end{equation*}
for some constants $C, \gamma > 0$. By passing to a further subsequence we may assume that $w_j \to \Phi$ locally uniformly in $\R^n \setminus \{ 0 \}$, for some function $\Phi$ which necessarily belongs to $H_{\alpha'}$. It follows that $\Phi$ is a solution of the equation
\begin{equation*}
F(D^2 \Phi) = |x|^{-2} (\alpha^* - \alpha') \Phi \mrbox{in} \R^n \setminus \{ 0 \}.
\end{equation*}
Notice that \EQ{est-wj} implies that $\max_{\partial B_1} |\Phi| = 1$, so $\Phi \not\equiv 0$. Since $(\alpha^* - \alpha') \Phi \geq 0$, the definition of $\alpha^*$ implies $\alpha' \leq \alpha^*$, and thus we deduce $\alpha' = \anom$. Therefore $\Phi$ is a solution of
\begin{equation*}
F(D^2\Phi) = 0\quad \mbox{in} \ \R^n \setminus \{ 0 \}.
\end{equation*}
By the strong maximum principle $\Phi \in \homclsplus{\anom}$. The uniqueness assertions follow from \PROP{HCP}, \COR{alpha-minus-maximum-principle} and the definition \EQ{def-anom} of $\anom$.
\end{proof}

Our construction of $\Phi$ in the case $\alpha^*(F) = 0$ is a variation of the above argument. It is complicated somewhat by the need to bend the approximate solutions so that their limit lies in the set $H_0^+$.

\begin{prop}
\label{prop:fundysol-case3}
If $\alpha^*(F) = 0$, then there exists a solution $\Phi \in \homclsplus{0}$ of the equation
\begin{equation*}
F(D^2 \Phi) = 0 \mrbox{in} \R^n \setminus \{ 0 \}.
\end{equation*}
Moreover, if $\beta > -1$ and $u \in H_\beta^+$ satisfies $F(D^2 u) = 0$ in $\R^n \setminus \{ 0 \}$, then $\beta = 0$ and $u \equiv \Phi + c$ for some $c \in \R$.
\end{prop}

\begin{proof}
Select a sequence $\ep_j \to 0$, $\ep_j>0$, and use \LEM{approx-fundysols} to find numbers $\alpha_j < 0$ satisfying
\begin{equation}\label{eq:bound-alpha-0}
-\ep_j < \alpha_j < - c \ep_j,
\end{equation}
and functions $u_j \in H_{\alpha_j}^+$ which satisfy the equation
\begin{equation}
\label{eq:pde-fp-alpha-zero}
F(D^2 u_j) = |x|^{-2} (- \ep_j - \alpha_j) u_j + \alpha_j |x|^{-\alpha_j-2} \mrbox{in} \R^n,
\end{equation}
and for which
\begin{equation} \label{eq:uj-bound}
\max_{\partial B_1} |u_j| = 1, \quad \mbox{so }\; 0<-u_j\le |x|^{-\alpha_j}.
\end{equation}
We may improve \EQ{bound-alpha-0} by observing that $- \ep_j - \alpha_j \leq \alpha_j$, as otherwise we would have $F(D^2 u_j) \leq 0$ in $\R^n \setminus \{ 0 \}$, in violation of \COR{alpha-minus-maximum-principle}. Thus we have
\begin{equation}\label{eq:alpha-j-like-ep-j}
- \frac{1}{2} \ep_j \leq \alpha_j \leq - c \ep_j,
\end{equation}
In particular, we have $\alpha_j \to 0$ as $j\to \infty$, and by taking a subsequence we may assume that the quantity
\begin{equation}\label{eq:limit-b}
- \alpha_j^{-1} (\ep_j + \alpha_j) \rightarrow b
\end{equation}
for some number $1 \leq b \leq (1-c)/c$.

The right-hand side of \EQ{pde-fp-alpha-zero} is locally bounded by $C\ep_j$. Recalling \EQ{uj-bound}, we may use the Harnack inequality to deduce that $|u_j + 1|$ is locally bounded by $C\ep_j$. That is,
\begin{equation}\label{eq:u-j-small}
\max_{K} |1+u_j | \leq C\ep_j,
\end{equation}
for all compact $K \subseteq \R^n \setminus \{ 0 \}$. Moreover, the H\"older estimates imply that
\begin{equation}\label{eq:u-j-small-holder}
\| 1+u_j \|_{C^\gamma(K)} \leq C \ep_j
\end{equation}
for any compact subset $K$ of $\R^n \setminus \{ 0 \}$.

Define $w_j := \alpha_j^{-1} \log( - u_j)$. It is straightforward to check that $w_j \in \homclsplus{0}$. Using \EQ{bound-alpha-0}, \EQ{alpha-j-like-ep-j}, \EQ{u-j-small}, and \EQ{u-j-small-holder}, it is simple to verify that
\begin{equation*}
\| w_j \|_{C^\gamma(K)} =\alpha_j^{-1}\| \log (1-(1+u_j)) \|_{C^\gamma(K)} \leq C.
\end{equation*}
By taking a further subsequence, we may assume that $w_j \rightarrow \Phi$ locally uniformly in $\R^n \setminus \{ 0 \}$ as $j \to \infty$, for a function $\Phi \in \homclsplus{0}$.

Formally differentiating, we find
\begin{equation*}
D^2 u_j =  \alpha_j u_j D^2 w_j - \alpha_j^2 u_j^2 D w_j \otimes D w_j.
\end{equation*}
Since $F$ is positively homogeneous and $\alpha_j u_j > 0$, a standard viscosity solution argument yields that $w_j$ is a solution of the equation
\begin{equation*}
F(D^2 w_j - \alpha_j u_j Dw_j \otimes Dw_j) = - \alpha_j^{-1}(\ep_j + \alpha_j) |x|^{-2} + u_j^{-1} |x|^{-\alpha_j - 2} \mrbox{in} \R^n \setminus \{ 0 \}.
\end{equation*}
Passing to the limit $j\to\infty$ in the viscosity sense, and recalling \EQ{limit-b} and \EQ{u-j-small}, we discover that $\Phi$ is a solution of the equation
\begin{equation*}
F(D^2\Phi) =  (b-1) |x|^{-2} \quad \mbox{in} \ \R^n \setminus \{ 0 \}
\end{equation*}
As $b \geq 1$ and $\anom=0$, we may apply \LEM{increment-alpha} to deduce that $b = 1$. The uniqueness assertions follow from \PROP{HCP}, \COR{alpha-minus-maximum-principle} and \EQ{def-anom}.
\end{proof}

\begin{proof}[Proof of \PROP{fundysol}]
\PROP{fundysol} is immediately obtained from Propositions \ref{prop:fundysol-case1}, \ref{prop:fundysol-case2}, and \ref{prop:fundysol-case3}.
\end{proof}

\begin{remark}
In light of \PROP{fundysol} and \COR{alpha-minus-maximum-principle}, the scaling exponent $\anom(F)$ may also be expressed as
\begin{multline}
\anom(F) = \min \left\{ \alpha \in (-1,\infty) : \ \mbox{there exists} \ v\in \homclsplus{\alpha}\ \right.\\
\left. \mbox{such that} \ F(D^2v) \leq 0 \ \mbox{in} \ \R^n \setminus \{ 0 \} \right\}.
\end{multline}

\end{remark}

\section{Examples and discussion} \label{sec:examples}

In this section we discuss several examples.

\begin{example}[Operators with radial fundamental solutions]\label{exam:radial-fundysols}
Felmer and Quaas \cite{Felmer:2009} observed that a certain class of operators $F$ have radial fundamental solutions. The key hypothesis is that $F$ is invariant with respect to orthogonal changes of coordinates, that is,
\begin{equation}\label{eq:Frotationallysymmetric}
F\left( Q^t M Q\right) = F(M) \quad\mbox{for every real orthogonal matrix} \ Q \ \mbox{and} \ M \in \Sy.
\end{equation}
In particular, they noticed that if $F$ satisfies \EQ{Frotationallysymmetric} and some additional hypotheses, then for some $\alpha > -1$,
\begin{equation*}
F(D^2\radfun\alpha ) = 0 \quad \mbox{in} \ \R^n \setminus \{0\}.
\end{equation*}
Of course, that \EQ{Frotationallysymmetric} suffices for $\Phi=\xi_{\anom}$ is an immediate consequence of \THM{liouville}.

Let us generalize this observation. Notice that \EQ{Frotationallysymmetric} is stronger than the condition
\begin{equation}\label{eq:Fsymmetricfundysol}
F(ay \otimes y - \iden) = F(a z\otimes z - \iden) \quad \mbox{for all} \ a\geq 1 \ \mbox{and} \ |y|=|z|=1.
\end{equation}
For fixed $|y|=1$ and $a \geq 1$, we see that
\begin{multline*}
 \lambda(n-1) - \Lambda (a-1) \leq \pucci(ay \otimes y - \iden)\leq F(ay \otimes y - \iden) \\ \leq \Pucci(ay \otimes y - \iden) = \Lambda(n-1) - \lambda (a-1).
\end{multline*}
Since $F$ is continuous, there exists a constant $1\leq \tilde{a} \leq \frac{\Lambda}{\lambda}(n-1) + 1$ such that
\begin{equation*}
F(\tilde ay \otimes y - \iden) = 0.
\end{equation*}
If \EQ{Fsymmetricfundysol} holds, then
\begin{equation*}
F(\tilde az \otimes z - \iden) = 0 \quad \mbox{for every} \ |z| =1.
\end{equation*}
It follows that
\begin{equation*}
F( D^2\radfun\alpha ) = 0 \quad \mbox{in} \ \R^n \setminus \{ 0 \},
\end{equation*}
for $\alpha = \tilde a -2$. Thus $\anom(F) = \tilde a - 2$ and $\Phi(F) = \radfun{\tilde a -2}$, and so we see that \EQ{Fsymmetricfundysol} implies that the fundamental solution of $F$ is radial.\end{example}

\begin{example}[Concave and Convex operators]
Let us review some well-known, elementary facts regarding the fundamental solutions of linear elliptic operators. The Laplacian $-\Delta$ has scaling exponent $\anom(-\Delta) = n-2$ and fundamental solution $\radfun{n-2}$. If $L$ is any linear, uniformly elliptic operator with constant coefficients, given by $Lu = - \sum_{i,j} a_{ij} u_{x_ix_j}$, then there is a change of coordinates with respect to which $L$ is transformed into $-\Delta$. It follows that $\anom(L) = n-2$ and the level sets of the fundamental solution $\Phi=\Phi(L)$ are ellipsoids. Moreover, it is clear that the level sets of $\Phi(L)$ distinguishes $L$ among all linear operators, up to a positive constant multiple.

From these facts, we will argue that if $F$ is a convex operator which is not the maximum of multiples of the same linear operator, then
\begin{equation*}
\anom(F) > n-2.
\end{equation*}
For such an operator $F$, there exists two linear operators $L_1$ and $L_2$ such that $L_1 \neq c L_2$ for every $c>0$, and
$
F \geq \max\{ L_1, L_2 \}.
$
Let $\Phi_1 : = \Phi(L_1)$ and $\Phi_2 : = \Phi(L_2)$ be the fundamental solutions for $L_1$ and $L_2$, respectively. Since $L_1$ and $L_2$ are not proportional, we see that $\Phi_1$ and $\Phi_2$ are not proportional, by the Liouville theorem. Since $F(D^2\Phi_i) \geq 0$ in $\R^n\setminus \{0 \}$ for $i=1,2$, \PROP{HCP} implies that $\anom(F) > n-2$.

A similar argument shows that if $F$ is concave and not the minimum of multiples of one linear operator, then $\anom(F) < n-2$. The underlying idea behind this example was previously observed in different contexts in \cite{Lions:1983,Armstrong:2009, Armstrong:preprint}.
\end{example}

It would be interesting to discover more about the relationship between the two scaling exponents $\anom(F)$ and $\anom(\tilde F)$. In particular, we do not know whether there is an operator $F$ satisfying (H1)-(H2) for which both $\anom(F)$ and $\anom(\tilde F)$ are negative. We show now that there is no such operator $F$ which also satisfies \EQ{Fsymmetricfundysol}.

\begin{prop}\label{prop:mean-anoms}
Suppose that $\Phi$ and $\tilde \Phi$ are radial functions. Then
\begin{equation}\label{eq:mean-anoms}
\frac{\lambda}{\Lambda}(n-2) \leq \max\{ \anom(F), \anom(\tilde F)\} \quad \mbox{and} \quad \min\{ \anom(F), \anom(\tilde F)\} \leq \frac{\Lambda}{\lambda}(n-2).
\end{equation}
\end{prop}
\begin{proof}
For ease of notation let us write $\alpha := \anom(F)$ and $\tilde\alpha := \anom(\tilde F)$. Since $\Phi = \radfun{\alpha}$ and $\tilde \Phi = \radfun{\tilde\alpha}$, we have that
\begin{equation*}
F( (\alpha+2) x\otimes x - \iden) = F( -(\tilde\alpha +2) x\otimes x + \iden) = 0 \quad \mbox{for all} \ x \in \partial B_1.
\end{equation*}
Let us suppose that $\max\{ \alpha,\tilde\alpha \} \leq k:= \lambda (n-2) / \Lambda \geq 0$. Select $x,y\in \partial B_1$ with $x\cdot y =0$ and observe that by (H1) we have
\begin{align*}
0 & = -2\Lambda k + 2\lambda (n-2) \\
& = \pucci\!\left( (k+2) x\otimes x +(k+2) y\otimes y - 2\iden \right) \\
& \leq \pucci\!\left( (\alpha+2) x\otimes x +(\tilde\alpha+2) y\otimes y - 2\iden \right) \\
& \leq F( (\alpha+2) x\otimes x - \iden) - F( -(\tilde\alpha +2) y\otimes y + \iden) \\
& = 0.
\end{align*}
It follows that $k = \alpha = \tilde \alpha$, and we obtain the first inequality in \EQ{mean-anoms}. The second inequality is obtained by a similar argument, using $\Pucci$ in place of $\pucci$ in the above calculation.
\end{proof}

\begin{remark}
The inequalities in \EQ{mean-anoms} are sharp. To see this, consider the operators
\begin{eqnarray*}
F_1(M)& :=& - \Lambda ( \mu_1(M) + \mu_n(M) ) - \lambda \sum_{i=2}^{n-1} \mu_i(M),\\
F_2(M) &:=& - \lambda ( \mu_1(M) + \mu_n(M) ) - \Lambda \sum_{i=2}^{n-1} \mu_i(M),
\end{eqnarray*}
where $\mu_1(M) \leq \mu_2(M) \leq \cdots \leq \mu_n(M)$ are the eigenvalues of $M \in \Sy$. It is easy to check that $F_1 = \tilde F_1$ and $F_2 = \tilde F_2$, that $F_1$ and $F_2$ satisfy (H1)-(H2) as well as \EQ{Frotationallysymmetric}, and
\begin{equation*}
\anom(F_1) = \frac{\lambda}{\Lambda} (n-2) \quad \mbox{and} \quad \anom(F_2) = \frac{\Lambda}{\lambda} (n-2).
\end{equation*}
\end{remark}

\section{Characterization of singularities and a Liouville theorem} \label{sec:singularities}

In this section we study the behavior near the origin of a solution $u\in C\left(B_1 \setminus \{ 0 \}\right)$ of the equation
\begin{equation}\label{eq:vanilla}
F\left( D^2u \right) = 0
\end{equation}
in $B_1 \setminus \{ 0 \}$ which is bounded on one side, and the behavior near infinity of a solution $u \in C(\R^n \setminus B_1)$ of \EQ{vanilla} in $\R^n \setminus B_1$ which is bounded on one side.

Throughout this section, we take $\anom = \anom(F)$ and $\Phi$ to be the scaling exponent and fundamental solution, respectively, for the operator $F$ obtained in \PROP{fundysol}, and $\tilde\alpha^* = \alpha^*(\tilde F)$ and $\tilde \Phi$ be the scaling exponent and upward-pointing fundamental solution, respectively, for the dual operator $\tilde F$ given by $\tilde F(M) := - F(-M)$.

We make repeated use of monotonicity properties of the quantities
\begin{equation}\label{eq:M-and-m}
m(r) : = \min_{\partial B_r} u \quad \mbox{and} \quad M(r) : = \max_{\partial B_r} u,\end{equation}\begin{equation}
\rho(r):= \min_{\partial B_r} \frac{u}{\Phi} \quad \mbox{and} \quad \bar\rho (r) : = \max_{\partial B_r} \frac{u}{\Phi},
\end{equation}
defined if $\Phi$ does not vanish on $\partial B_r$.

\begin{lem}\label{lem:harnacksect5}
Suppose $u$ is a solution of \EQ{vanilla} in $B_1\setminus\{0\}$ (resp. in $\R^n\setminus B_1$). Then there exists a constant $C=C(n,\lambda,\Lambda)$ such that for each $r\in (0,1/2)$ (resp. $r\in(2,\infty)$) we have
$$
M(r)\le C m(r) \qquad \mbox{and}\qquad \bar\rho(r)\le C^2\rho(r).
$$
\end{lem}
\begin{proof}
This is a simple consequence of the Harnack inequality and the fact that if a function $x\to u(x)$ is a solution of $F\left( D^2u \right) = 0$, then so is $x\to u(x/r)$. \end{proof}

\subsection{Classification of isolated singularities} Our proof of \THM{singularities}, though considerably more complicated,  borrows some ideas from Labutin \cite{Labutin:2001}, who proved \THM{singularities} for $F=\pucciSub$, in the case that $\anom(\pucciSub) \geq 0$. The idea is to show that either the singularity at the origin of a solution $u$ of
\begin{equation} \label{eq:inside}
F(D^2u) = 0 \quad \mbox{in} \ B_1 \setminus \{ 0 \}
\end{equation}
is removable, or else is bounded between two multiples of $\Phi$ near the origin. Then we use the Harnack inequality and the strong maximum principle to squeeze the gap as we blow up the function $u$ at the origin. A similar idea will establish the corresponding conclusions in the case that $\anom(F) < 0$.

We divide the proof of \THM{singularities} into five lemmas.
The first step  is the following result, which states that a nonnegative solution $u$ of \EQ{vanilla} must either be  bounded near the origin, or $\anom(F) \geq 0$ and $\rho(r) \geq c > 0$.

\begin{lem}\label{lem:sing-bound-above}
Assume that $u \in C\left(B_1 \setminus \{ 0 \}\right)$ is a nonnegative solution of \EQ{inside}. Suppose that either (i) $\anom(F) \geq 0$ and $\liminf_{r\to 0} \rho (r) = 0$, or (ii) $\anom(F) < 0$. Then $u$ is bounded in $B_{1/2} \setminus \{ 0 \}$.
\end{lem}
\begin{proof}
We first consider the case (i). By adding a constant to $\Phi$ in the case that $\anom(F) = 0$, we may assume that $\Phi > 0$ in $B_1$. Let $r_k\to 0$ be such that $\rho(r_k)\to 0$ as $k \to \infty$. Select a point $x_k \in B_{r_k} \setminus \{ 0 \}$ such that
\begin{equation*}
u(x_k) \leq (1/k) \Phi(x_k).
\end{equation*}
According to the Harnack inequality, there exists a constant $C> 0$, depending only on $n$ and the ellipticity constants $\Elp$ and $\elp$, such that
\begin{equation*}
u(x) \leq C (1/k) \Phi(x) \quad \mbox{for all} \  |x| = |x_k|.
\end{equation*}
According to the maximum principle,
\begin{equation*}
u \leq C(1/k) \Phi + M(1/2) \quad \mbox{in} \ \ B_{1/2} \setminus B_{r_k}.
\end{equation*}
Passing to the limit $k \to \infty$, we obtain
\begin{equation}\label{eq:u-bounded-p-ball}
u \leq M(1/2) \quad \mbox{in} \ \  B_{1/2} \setminus \{ 0 \}.
\end{equation}
Thus $u$ is bounded above in the punctured ball $B_{1/2} \setminus \{ 0 \}$.

We now consider case that $\anom(F) < 0$. By Lemma \LEM{harnacksect5} we know that  $M(r) \leq C m(r)$ for every $0 < r < 1/2$. Define
\begin{equation*}
a:= -\max_{\partial B_{1/2}} \Phi > 0 \quad \mbox{and} \quad \psi(x):= a^{-1}(\Phi(x) + a).
\end{equation*}
Then $\psi \leq 0$ on $\partial B_{1/2}$, $\psi \leq 1$ on $\bar B_{1/2}$, and $\psi > 1/2$ in a neighborhood of the origin, say in $B_{r_0}\setminus\{0\}$. It is clear that $F(D^2\psi) = 0$ in $B_{1/2} \setminus \{ 0 \}$. According to the maximum principle,
\begin{equation*}
u \geq m(r) \psi \quad \mbox{in} \ B_{1/2} \setminus \bar B_r.
\end{equation*}
Hence $u \geq C^{-1} M(r) \psi$ in $B_{r_0} \setminus \bar B_r$. Since $\psi > 0$ near the origin, it follows that
\begin{equation*}
\sup_{0 < r < r_0} M(r) < \infty.
\end{equation*}
Therefore, $u$ is bounded in $B_{1/2} \setminus \{ 0 \}$.
\end{proof}

The next auxiliary result says that under the conclusion of the previous lemma we can define $u$ as a continuous function at the origin.

\begin{lem}\label{lem:sing-bound-to-cont}
Suppose that $u \in C\left(B_1 \setminus \{ 0 \}\right)$ is a bounded solution of \EQ{inside}. Then $u$ can be defined at the origin so that $u \in C(B_1)$.
\end{lem}
\begin{proof}
We must show that $\lim_{x\to 0} u(x)$ exists. Define
\begin{equation*}
u_0 := \liminf_{x\to 0} u(x).
\end{equation*}
Choose $\ep > 0$. Define $v(x): = u(x) - u_0 + \ep$, and fix $0 < r = r(\ep) < 1/2$ so small that $v > 0$ in $B_r \setminus \{ 0 \}$. By making $r$ smaller, if necessary, we can find $x_1 \in \partial B_r$ such that $v(x_1) \leq 2\ep$. Let $0 < s < r$. Choose $x_2 \in B_s \setminus \{ 0 \}$ such that $v(x_2) \leq 2\ep$. By the Harnack inequality, there exists a constant $C=C(\Elp,\elp,n) > 0$ such that
\begin{equation*}
v \leq C\ep \quad \mbox{on} \ \partial B_r \cup \partial B_{|x_2|}.
\end{equation*}
 By the maximum principle, $v \leq C\ep$ in $B_r \setminus B_s$. We send $s \to 0$ to deduce that
\begin{equation*}
v \leq C\ep \quad \mbox{in} \ B_r \setminus \{ 0 \}.
\end{equation*}
Thus
\begin{equation*}
\sup_{B_r \setminus \{ 0 \} } u \leq u_0 + C\ep.
\end{equation*}
It follows that $\limsup_{x\to 0} u(x) \leq u_0$.
\end{proof}

\begin{lem} \label{lem:removable-plus}
Assume that $u \in C(B_1)$ is a solution of \EQ{inside} such that
$u(0) = 0$. Suppose that either (i) $\anom(F) \geq 0$, or (ii)
$\anom(F) < 0$ and $\liminf_{r\to 0} \rho(r) \leq 0$. Then $u$ is a
$\mathrm{subsolution}$ of the equation $F(D^2u) = 0$ in the whole ball $B_1$.
\end{lem}
\begin{proof}
Consider a smooth test function $\varphi$ for which the function
$u-\varphi$ has a strict local maximum at the origin. We must show
that
\begin{equation}\label{eq:removable-plus-wts}
F(D^2\varphi(0)) \leq 0.
\end{equation}
We may assume without loss of generality that $\varphi(0) = 0 = u(0)$.

We may also assume without loss of generality that $D\varphi(0) = 0$.
To see this, define
\begin{equation*}
\tilde u(x):= u(x) - x\cdot D\varphi(0), \quad \mbox{and} \quad \tilde
\varphi(x):= \varphi(x) - x \cdot D\varphi(0),
\end{equation*}
and notice that $\tilde u - \tilde \varphi$ has a strict local maximum
at the origin, $D\tilde \varphi(0) = 0$, $\tilde \varphi(0) = 0 =
\tilde u(0)$, and $\tilde u$ is a solution of \EQ{inside}. Moreover,
our hypotheses (i) or (ii) hold for $\tilde u$. To get the second
condition in case (ii), notice that
\begin{equation*}
\liminf_{|x|\to 0} \frac{\tilde u(x)}{\Phi(x)} = \liminf_{|x|\to 0}
\frac{u(x) - x\cdot D\varphi(0)}{\Phi(x)} \leq \liminf_{|x| \to 0}
\frac{u(x)}{\Phi(x)} + C \lim_{|x|\to 0} \frac{|x|}{|x|^{-\anom}}
\end{equation*}
The last expression on the right vanishes, since $\anom > -1$.
Finally, we remark that our conclusion holds for $\tilde u$ if and
only if it holds for $u$. Therefore, we may assume that $u = \tilde u$
and $D\varphi(0) = 0$.

We claim that
\begin{equation}\label{eq:sing-3-wts}
0 \leq \max_{\partial B_r} u \quad \mbox{for every} \ r>0.
\end{equation}
In the case that (i) holds, we argue just as we did to obtain
\EQ{u-bounded-p-ball} in the proof of \LEM{sing-bound-above}, by replacing $u$ by $u+C$, where $C$ is chosen so that $u+C$ is positive in $B_1$, and then showing that
$
u+C\le \max_{\partial B_r} u+C$ in $B_r$. Next, consider
the case that (ii) holds, and suppose on the contrary that
$\max_{\partial B_r} u < 0$. By multiplying $u$ by a positive
constant, we may assume that
\begin{equation*}
u\leq \Phi \quad \mbox{on} \ \partial B_r.
\end{equation*}
Since $u(0) = \Phi(0)$, the maximum principle implies that $u \leq
\Phi$ in $B_r \setminus \{ 0 \}$. This contradicts the second
hypothesis in (ii). We have established \EQ{sing-3-wts}.

Owing to \EQ{sing-3-wts}, there exists a unit vector $z\in \partial
B_1$ and a sequence $\{ y_j\} \subseteq B_1 \setminus \{ 0 \}$ such
that $y_j \to 0$ and
\begin{equation*}
u(y_j) \geq 0 \quad \mbox{and}  \quad  z\cdot y_j > \frac{|y_j|}{2}
\quad \mbox{for all} \ j.
\end{equation*}
For $\ep > 0$, we define
\begin{equation*}
\psi^\ep(x) : = \varphi(x) - \ep z\cdot x.
\end{equation*}
Select $r,\delta > 0$ sufficiently small that
\begin{equation*}
u(x) - \varphi(x) \leq - \delta \quad \mbox{for every} \ \ |x|=r.
\end{equation*}
For $\ep > 0$ small enough, we have
\begin{equation*}
u(0) = \psi^\ep(0) = 0 \ \ \mbox{and} \ \ u(x) - \psi^\ep(x) \leq
-\frac{\delta}{2} \quad \mbox{for all} \ \ |x|=r.
\end{equation*}
Notice that
\begin{equation*}
\psi^\ep(y_j) = \varphi(y_j) - \ep z \cdot y_j \leq -\frac{\ep}{2}
|y_j| + o\left( |y_j|\right) \quad \mbox{as} \ j\to\infty.
\end{equation*}
Thus for $j$ large enough, we have $u(y_j) - \psi^\ep(y_j) \geq
-\psi^\ep(y_j) > 0$, as well as $|y_j| < r$. Let $x_\ep \in B_r$ such
that
\begin{equation*}
u(x_\ep) - \psi^\ep(x_\ep) = \max_{B_r} (u-\psi^\ep).
\end{equation*}
Since $x_\ep \neq 0$, we deduce that
\begin{equation*}
F(D^2\varphi(x_\ep)) = F(D^2\psi^\ep(x_\ep)) \leq 0.
\end{equation*}
It is clear that $x_\ep \to 0$ as $\ep \to 0$. We now pass to limits
to obtain \EQ{removable-plus-wts}.
\end{proof}
\begin{lem}\label{lem:sing-like-plus}
Suppose that $\anom(F) \geq 0$, and $u \in C\left(B_1 \setminus \{ 0 \}\right)$ is a nonnegative solution of \EQ{inside}. Then
\begin{equation} \label{eq:ratio-not-exploding}
\limsup_{r\to 0} \rho(r) < \infty.
\end{equation}
Moreover, if
\begin{equation}\label{eq:ratio-not-vanishing}
a:= \liminf_{r\to 0} \rho(r) > 0,
\end{equation}
then there is a constant $C > 0$ such that
\begin{equation}\label{eq:ratio-limit}
a\Phi - C \leq u(x) \leq a \Phi + C \quad \mbox{in} \ B_{1/2} \setminus \{ 0 \}.
\end{equation}
\end{lem}
\begin{proof}
In the case that $\anom(F) = 0$, we may assume that $\Phi > 0$ in $B_1 \setminus \{ 0 \}$. The maximum principle implies that for any $0 < r < 1/2$,
\begin{equation*}
u \geq \rho(r) \left( \Phi - \max_{\partial B_{1/2}} \Phi \right) \quad \mbox{in} \ B_{1/2} \setminus B_r.
\end{equation*}
Since $\max_{\partial B_{1/4}} \Phi > \max_{\partial B_{1/2}} \Phi$ and $\max_{\partial B_{1/4}} u < \infty$, we deduce that
\begin{equation*}
\sup_{0< r < 1/4} \rho(r) < \infty.
\end{equation*}
In particular, \EQ{ratio-not-exploding} holds, and the Harnack inequality implies that
\begin{equation}\label{eq:ratio-upper-bound}
\bar a: = \limsup_{r\to 0} \bar \rho(r) < \infty.
\end{equation}

Suppose now that $a:= \liminf_{r\to 0} \rho(r) > 0$. By adding a positive constant to $u$, we may assume that $\rho(1/2) \geq 2\bar a$. We claim that for sufficiently small $r>0$,
\begin{equation}\label{eq:sing-case1-wts}
\rho(r) = \min_{\bar B_{1/2} \setminus B_r} \frac{u}{\Phi}.
\end{equation}
Select $0 < r_0 < 1/2$ small enough that
\begin{equation*}
\sup_{B_{r_0} \setminus \{ 0 \}} \frac{u}{\Phi}  \leq \frac{3}{2} \bar a.
\end{equation*}
Then for $0 < r < r_0$, we have $u \geq \rho(r) \Phi$ on $\partial B_{1/2} \cup \partial B_{r}$. By the maximum principle, $u \geq \rho(r) \Phi$ on $\bar B_{1/2} \setminus B_r$. Hence \EQ{sing-case1-wts} holds for every $r \in (0,r_0)$. We deduce that $r \mapsto \rho(r)$ is increasing on the interval $(0,r_0)$, and thus $\lim_{r\downarrow 0} \rho(r) = a$. In particular,
\begin{equation} \label{eq:u-bound-below-sing}
u \geq a \Phi \quad \mbox{in} \ B_{1/2} \setminus \{ 0 \}.
\end{equation}

For each $0 < r < r_0$, select $x_r$ with $|x_r| = r$ such that $u(x_r) = \rho(r) \Phi(x_r)$. We now employ a rescaling argument to show that $a = \bar a$. That is, we claim that
\begin{equation}\label{eq:blow-up-sing-wts}
\lim_{x\to 0} \frac{u(x)}{\Phi(x)} = a.
\end{equation}
To prove \EQ{blow-up-sing-wts} we consider the cases $\anom(F) > 0$ and $\anom(F) =0$ separately.

Suppose first that $\anom(F) > 0$. For each $0 < r < r_0$ and $x \in B_{1/(2r)}$, we define
\begin{equation*}
v_r(x) : = r^\anom u(rx).
\end{equation*}
Recalling \EQ{ratio-upper-bound} and \EQ{u-bound-below-sing}, for every compact set $K \subseteq \R^n \setminus \{ 0 \}$ we have the estimate
\begin{equation*}
\sup_{0<r<r_0} \| v_r \|_{L^\infty(K)} \leq C_K.
\end{equation*}
Thus using the H\"older estimates, we can find a function $v\in C(\R^n \setminus \{ 0 \})$ and a sequence $r_j \to 0$ such that
\begin{gather}
v_{r_j} \to v \quad \mbox{locally uniformly in} \ \R^n \setminus \{ 0 \} \ \mbox{as} \ j \to \infty.
\end{gather}
By taking a further subsequence, we may also assume that $r^{-1}_j x_{r_j} \to y$ as $j \to \infty$ for some $y \in \partial B_1$. According to \EQ{u-bound-below-sing}, we have $v \geq a \Phi$. It is clear that $v$ is a solution of $F(D^2v) = 0$ in $\R^n \setminus \{ 0 \}$. Since
\begin{equation*}
v(y) = \lim_{j\to \infty} r_j^{\anom} u(x_{r_j}) = \lim_{j\to \infty} r_j^{\anom} \rho(r_j) \Phi(x_{r_j}) = \lim_{j\to \infty} \rho(r_j) \Phi(x_{r_j}/r_j) = a \Phi(y),
\end{equation*}
the strong maximum principle implies that $v \equiv a \Phi$. From this we deduce that the full sequence $\{v_r\}_{r>0}$ converges to the function $a \Phi$ locally uniformly as $r \downarrow 0$. Thus
\begin{equation*}
\limsup_{x\to 0} \frac{u(x)}{\Phi(x)} = \limsup_{r\to 0} \max_{x\in \partial B_r} \frac{u(x)}{r^{-\anom} \Phi(x/r)} = \limsup_{r\to 0} \max_{x\in \partial B_1} \frac{v_r(x)}{\Phi(x)} = a.
\end{equation*}
This verifies \EQ{blow-up-sing-wts} in the case $\anom(F) > 0$.

Now suppose that $\anom(F) = 0$. For each $0 < r < r_0$, define the function
\begin{equation*}
v_r(x) : = \frac{u(rx)}{\Phi(x_r)}, \quad x\in B_{1/2r} \setminus \{ 0 \}.
\end{equation*}
It is clear that $v_r$ satisfies the equation $F(D^2v) = 0$ in $B_{1/2r} \setminus \{ 0 \}$. To get a lower bound for $v_r$, we notice that
\begin{equation*}
v_r(x) \geq a \frac{\Phi(rx)}{\Phi(x_r)} = a \frac{\Phi(x) - \log r}{\Phi(x_r/r) - \log r} \rightarrow a \quad \mbox{locally uniformly as} \ r \downarrow 0.
\end{equation*}
Since
\begin{equation*}
v_r\left( \frac{x_r}{r} \right) = \rho(r),
\end{equation*}
the Harnack inequality provides the bound $\| v_r \|_{L^\infty(K)} \leq C_K$ for every $0 < r < r_1 \leq r_0$ and compact subset $K \subseteq B_{1/(2r_1)} \setminus \{ 0 \}$. As before, using the H\"older estimates we can find a subsequence $r_j \downarrow 0$, a point $y \in \partial B_1$, and a function $v\in C(\R^n \setminus \{ 0 \})$ for which
\begin{equation*}
v_{r_j} \rightarrow v \ \mbox{locally uniformly in} \ \R^n \setminus \{ 0 \} \ \mbox{and} \ r^{-1}_j x_{r_j} \rightarrow y \ \mbox{as} \ j \to \infty.
\end{equation*}
We immediately deduce that
\begin{equation*}
F(D^2v) = 0 \quad \mbox{in} \ \R^n \setminus \{ 0 \},
\end{equation*}
as well as $v(y) = a$ and $v \geq a$ in $\R^n \setminus \{ 0 \}$. The strong maximum principle implies $v \equiv a$. Therefore,
\begin{align*}
\limsup_{x\to 0} \frac{u(x)}{\Phi(x)} & = \limsup_{r\to 0} \max_{x\in \partial B_1} \frac{u(rx)}{\Phi(rx)}  \\
& = \limsup_{r\to 0} \max_{x\in \partial B_1} \frac{v_r(x) \Phi(x_r)}{\Phi(rx)} \\
& = \limsup_{r\to 0} \max_{x\in \partial B_1} \frac{v_r(x) \left( \Phi(x_r/r) - \log r \right)}{\Phi(x) - \log r} \\
& = a.
\end{align*}
This completes the proof of \EQ{blow-up-sing-wts}. In particular, $\bar a = a$.

We have shown above that by adding a constant to $u$ so that $\rho(1/2) \geq 2 \bar a = 2a$, then we deduce that $u \geq a\Phi$ in $B_{1/2} \setminus \{ 0 \}$. By a symmetric argument, we can show that by subtracting a constant from $u$ so that $\bar \rho(1/2) \leq a / 2$, then $u \leq \bar a\Phi = a\Phi$ in $B_{1/2} \setminus \{ 0 \}$. Therefore \EQ{ratio-limit} holds.
\end{proof}

\begin{lem}\label{lem:sing-like-minus}
Assume that $\anom(F) < 0$ and $u \in C(B_1)$ is a solution of \EQ{inside} such that $u(0) = 0$ and
\begin{equation}\label{eq:sing-case3-eq1}
a : = \liminf_{r\to 0} \rho(r) > 0.
\end{equation}
Then $0 < a < \infty$, and
\begin{equation}
\lim_{x\to 0} \frac{u(x)}{ \Phi(x)} = a.
\end{equation}
\end{lem}
\begin{proof}
Our hypothesis \EQ{sing-case3-eq1} implies that there exists $0 < r_0 < 1$ such that $\rho(r) > 0$ for $0 < r \leq  r_0$. For such $r$, since $u \leq \rho(r) \Phi$ on $\partial B_r \cup \{ 0 \}$, the maximum principle implies that $u \leq \rho(r)\Phi$ on $\bar B_r$. In particular, $\rho(r) = \min_{\bar B_r} u / \Phi$ and $u< 0$ near the origin. It follows that the map $r\mapsto \rho(r)$ is decreasing in $r \in (0,r_0)$ and $\lim_{r\downarrow 0} \rho (r) = a$. By a similar argument, we see that the map $r\mapsto \bar\rho(r)$ satisfies $\bar \rho(r) = \max_{\bar B_r} u/\Phi$ for all $0 < r < r_0$, and is therefore increasing in $r \in (0,r_0)$. In particular, $a \leq \limsup_{r\to 0} \bar \rho(r) < \bar\rho(r_0) < \infty$.

For every $0 < r < r_0$, select $x_r \in \partial B_r$ such that $u(x_r) = \rho(r) \Phi(x_r)$. Define the function
\begin{equation*}
v_r (x) : = r^{\anom} u(rx), \quad 0 < r < r_0, \ x \in \bar B_{1/2r}.
\end{equation*}
By the homogeneity of $\Phi$, for sufficiently small $s> 0$ we have
\begin{equation}\label{eq:sing-help-case2}
\bar\rho(s) \Phi \leq v_r \leq \rho(s) \Phi \quad \mbox{in} \ \bar B_{s/r},
\end{equation}
as well as $v_r \leq \rho(r) \Phi$ on $\bar B_1$, and $v_r(x_r / r) = \rho(r) \Phi(x_r /r)$. Using the H\"older estimates, we can find a subsequence $r_j \downarrow 0$ for which
\begin{equation*}
v_{r_j} \to v \quad \mbox{locally uniformly in} \ \R^n \setminus \{ 0 \},
\end{equation*}
and $x_{r_j} / r_j \to y$ for some $v\in C(\R^n \setminus \{ 0 \})$ and $y \in \partial B_1$. Passing to limits we deduce that $v$ is a solution of the equation
\begin{equation*}
F(D^2v) = 0 \quad \mbox{in} \ \R^n \setminus \{ 0 \}.
\end{equation*}
According to \EQ{sing-help-case2}, we have $v \leq a \Phi$ in $\R^n \setminus \{ 0 \}$ and $v(y) = a \Phi(y)$. By the strong maximum principle, $v \equiv a \Phi$. It follows that the full sequence $\{ v_r \}_{r>0}$ converges to $a\Phi$ locally uniformly as $r \downarrow 0$. Thus
\begin{equation*}
\limsup_{x\to 0} \frac{u(x)}{\Phi(x)} = \limsup_{r\to 0} \max_{x\in \partial B_r} \frac{u(x)}{r^{-\anom} \Phi(x/r)} = \limsup_{r\to 0} \max_{\partial B_1} \frac{v_r}{\Phi} = a.
\end{equation*}
The proof is complete.
\end{proof}

We now combine the previous five lemmas into a proof of \THM{singularities}.

\begin{proof}[Proof of \THM{singularities}]
Let us assume that $u$ is bounded below in a neighborhood of the origin and  first consider the case when
\begin{equation*}
\liminf_{r\to 0} \rho(r) \leq 0.
\end{equation*}
According to Lemmas \ref{lem:sing-bound-above}, \ref{lem:sing-bound-to-cont}, and \ref{lem:removable-plus}, we can define $u$ at the origin so that $u \in C(B_1)$, and $u$ is a subsolution of \EQ{vanilla} in the whole ball $B_1$. If $\anom(\tilde F) \geq 0$, or if $\anom(\tilde F) < 0$ and $\liminf_{x\to 0} (-u(x) + u(0) ) / \tilde\Phi(x) \leq 0$, then applying \LEM{removable-plus} to $-u$ we see that $u$ is supersolution of \EQ{vanilla} in the whole ball, and therefore the singularity is removable, giving us alternative (i). In the case that $\anom(\tilde F) < 0$ and $a: = \liminf_{x\to 0} (-u(x) + u(0) ) / \tilde\Phi(x) > 0$, then \LEM{sing-like-minus} implies that alternative (v) holds.

On the other hand, if
\begin{equation*}
a: = \liminf_{r\to 0} \rho(r) > 0,
\end{equation*}
then according to Lemmas \ref{lem:sing-like-plus} and \ref{lem:sing-like-minus}, we have $a<\infty$, and $\anom(F) \geq 0$ implies that alternative (ii) holds, while $\anom(F) < 0$ implies that alternative (iv) holds.

This completes the proof in the case that $u$ is bounded below. If $u$ is bounded above, then we repeat our argument with $-u$ in place of $u$, and $\tilde F$ in place of $F$.
\end{proof}

\subsection{Classification of singularities at infinity}

In this subsection we study the behavior near infinity of a solution $u \in C(\R^n \setminus B_1)$ of the equation
\begin{equation} \label{eq:outside}
F(D^2u) = 0 \quad \mbox{in} \ \R^n \setminus B_1.
\end{equation}
Our approach mirrors the proof of \THM{singularities} given in the previous subsection.

\begin{lem}\label{lem:outside-lem-1}
Assume that $u \in C(\R^n \setminus B_1)$ is a nonpositive solution of equation \EQ{outside}, and that either (i) $\anom(F) > 0$, or (ii) $\anom(F) \leq 0$ and $\liminf_{r\to \infty} \rho(r) = 0$. Then $u$ is bounded in $\R^n \setminus B_1$.
\end{lem}
\begin{proof}
According to the Harnack inequality and the homogeneity of $F$, there exists a constant $0 < C< 1$ such that $M(r) \leq C m(r)$ for all $r \geq 2$. Thus it suffices to show that $M(r)$ is bounded below.

We first consider case (i). Recall we assume that $\min_{\partial B_1} \Phi = 1$. By the maximum principle, for every $r>1$ we have
\begin{equation*}
u(x) \leq -M(r) \Phi(x) + M(r) \quad \mbox{for all} \ x \in B_r \setminus B_1.
\end{equation*}
Evaluating this expression at a point $|x|=r_0$ such that $\Phi(x) <1/2 $ if $|x|\ge r_0$, we discover that
\begin{equation*}
m(r_0) \leq u(x) \leq (1/2) M(r) \quad \mbox{for all} \ r>2,
\end{equation*}
verifying that $M(r)$ is bounded below.

We now consider case (ii). By subtracting a positive constant from $\Phi$ in the case $\alpha=0$, we may assume that $\Phi < 0$ in $\R^n \setminus B_1$. The Harnack inequality implies that $\bar\rho(r) \leq C\rho(r)$ for all $r\geq 2$. Thus for any $\ep > 0$, there exists $r> 2$ such that $\bar\rho(r) < \ep$. By the maximum principle, we have
\begin{equation*}
u(x) \geq \ep \Phi(x) + m(1) \quad \mbox{in} \ B_r \setminus B_1.
\end{equation*}
Let $r \to \infty$ and then $\ep \to 0$ to deduce that $u(x) \geq m(1)$ in $\R^n \setminus B_1$.
\end{proof}

\begin{lem} \label{lem:outside-lem-2}
Suppose that $u \in C(\R^n \setminus B_1)$ is a bounded solution of \EQ{outside}. Then $\lim_{|x|\to \infty} u(x)$ exists.
\end{lem}
\begin{proof}
Let $u_0 := \liminf_{|x|\to \infty} u(x)$. Let $\ep > 0$, and define $v(x) := u(x) - u_0 + \ep$. If we take $r> 1$ very large, then $v>0$ in $\R^n \setminus B_r$. We can find a point $x_1 \in \R^n \setminus B_r$ such that $v(x_1) \leq 2\ep$. For any $s > |x_1|$, there is a point $x_2 \in \R^n \setminus B_s$ such that $v(x_2) \leq 2 \ep$. By the Harnack inequality, there is a constant $C=C(n,\elp,\Elp)$ such that
\begin{equation*}
v \leq C \ep \quad \mbox{on} \ \partial B_{|x_1|} \cup \partial B_{|x_2|}.
\end{equation*}
By the maximum principle,
\begin{equation*}
v\leq C \ep \quad \mbox{in} \ B_s \setminus B_{|x_1|}.
\end{equation*}
Letting $s\to \infty$, we deduce that $v\leq C\ep$ in $\R^n \setminus B_{|x_1|}$. Therefore, $\limsup_{|x| \to \infty} v(x) \leq C\ep$. This implies that $\limsup_{|x| \to \infty} u(x) \leq u_0 + C\ep$.
\end{proof}

\begin{lem} \label{lem:outside-lem-3}
Assume that $u \in C(\R^n \setminus B_1)$ is a bounded solution of \EQ{outside} and $\lim_{|x| \to \infty} u(x) = 0$. Suppose that either (i) $\anom(F) > 0$ and $\liminf_{r \to\infty} \rho(r) \leq 0$, or (ii) $\anom(F) \leq 0$. Then $m(r) \leq 0$ for all $r>1$.
\end{lem}
\begin{proof}
Suppose on the contrary that (i) holds but $m(r) > 0$ for some $r>1$. Let $c>0$ be so small that $c\Phi \leq m(r)$ on $\partial B_r$. By the maximum principle, for any $\ep > 0$ we have
\begin{equation*}
u \geq c\Phi - \ep \quad \mbox{in} \ \R^n \setminus B_r.
\end{equation*}
Thus $u \geq c\Phi$, a contradiction to our assumption that $\liminf_{r\to \infty} \rho(r) \leq 0$. This completes the proof in case (i). In the case that (ii) holds, we argue as in the last paragraph in the proof of \LEM{outside-lem-1}.
\end{proof}

\begin{lem}\label{lem:outside-lem-4}
Suppose that $\anom(F) \leq 0$ and $u \in C(\R^n \setminus B_1)$ is a nonpositive solution of \EQ{outside}. Then $\limsup_{r\to \infty} \rho(r) < \infty$. Moreover, if $a:= \liminf_{r\to \infty} \rho(r) > 0$, then there exists a constant $C > 0$ such that
\begin{equation} \label{eq:blow-out-stuck}
a\Phi(x) \leq u(x) \leq a\Phi + C \quad \mbox{in} \ \R^n \setminus B_2.
\end{equation}
\end{lem}
\begin{proof}
In the case that $\anom(F) = 0$, we may assume that $\Phi < 0$ in $\R^n \setminus B_1$. The maximum principle implies that for any $r>2$,
\begin{equation*}
u(x) \leq \rho(r) \left( \Phi - \min_{\partial B_2} \Phi \right) \quad \mbox{in} \ B_r \setminus B_2.
\end{equation*}
In particular, $m(r_0) \leq \rho(r) \left( \max_{\partial B_{r_0}} \Phi -  \min_{\partial B_2} \Phi \right)$ for all $r_0>2$. Since we have $\min_{\partial B_{r_0}} \Phi < \min_{B_2} \Phi$ if $r_0$ is sufficiently large, we deduce that $\sup_{r > r_0} \rho(r) < \infty$. The Harnack inequality implies that
\begin{equation*}
\bar a:= \limsup_{r\to \infty} \bar\rho(r) < \infty.
\end{equation*}

Suppose now that $a:= \liminf_{r\to \infty} \rho(r) > 0$. By subtracting a positive constant from $u$, we may assume that $\rho(2) \geq 2a$. By the maximum principle,
\begin{equation*}
u \leq \rho(r) \Phi \quad \mbox{in} \ \bar B_r \setminus B_2
\end{equation*}
for all $r>2$ such that $\rho(r) < 2a$. Sending $r \to \infty$, we find that $u \leq a \Phi$ in $\R^n \setminus \{ 0 \}$. A scaling argument very similar to the one in the proof of \LEM{sing-like-plus} confirms that
\begin{equation*}
\lim_{|x| \to \infty} \frac{u(x)}{\Phi(x)} = a.
\end{equation*}
In particular, $\bar a = a$. By adding a constant to $u$ so that $\bar \rho(2) \leq \frac{1}{2} a$, we find that $u \geq a\Phi$ in $\R^n \setminus B_2$, by the maximum principle. Thus we have \EQ{blow-out-stuck}.
\end{proof}

\begin{lem}\label{lem:outside-lem-5}
Assume that $\anom(F) > 0$ and $u\in C(\R^n \setminus B_1)$ is a solution of \EQ{outside} such that $0 = \lim_{|x| \to \infty} u(x)$, and $a: = \liminf_{r\to \infty} \rho(r) > 0$. Then $a< \infty$ and $\lim_{r \to \infty} \bar \rho(r) = a < \infty$.
\end{lem}
\begin{proof}
Notice that $\bar \rho (r)$ is decreasing, since the maximum principle implies that for all $\ep > 0$,
\begin{equation*}
u \leq \bar \rho(r) \Phi + \ep \quad \mbox{in} \ \R^n \setminus B_r.
\end{equation*}
Thus $a \leq \bar a:= \limsup_{r \to \infty} \bar\rho(r) \leq \bar \rho (1)$. Set $v_r (x) = r^{\anom}u(rx)$. By a rescaling argument similar to that in \LEM{sing-like-minus}, we find that $v_r \to a \Phi$ locally uniformly as $r \to \infty$. It follows that $a = \bar a$.
\end{proof}

\begin{proof}[Proof of \THM{singularity-infinity}]
The proof is nearly identical to the proof of \THM{singularities}, with Lemmas \ref{lem:outside-lem-1}, \ref{lem:outside-lem-2}, \ref{lem:outside-lem-3}, \ref{lem:outside-lem-4}, and \ref{lem:outside-lem-5} taking the place of Lemmas \ref{lem:sing-bound-above}, \ref{lem:sing-bound-to-cont}, \ref{lem:removable-plus}, \ref{lem:sing-like-plus}, and \ref{lem:sing-like-minus}, respectively.
\end{proof}

\subsection{A Liouville-type result}

In this subsection we use Theorems \ref{thm:singularities} and \ref{thm:singularity-infinity} to prove \THM{liouville}.

\begin{proof}[Proof of \THM{liouville}]
We proceed by considering each of the alternatives provided by \THM{singularities}.

Case (i): the singularity at the origin is removable. In this case, the function $u$ is a solution of $F(D^2u) = 0$ in the whole space $\R^n$, and $u$ is bounded from above or below in $\R^n$. It now follows from the Liouville theorem for uniformly elliptic equations that $u$ is constant. (The Liouville theorem is an immediate consequence of the Harnack inequality, see Remark 4 in Chapter 4 of \cite{Caffarelli:Book}).

Case (ii)(a): $\anom(F) > 0$ and $u(x) = a \Phi(x) + O(1)$ as $x \to 0$. We may assume that $a = 1$. We claim that $u$ is bounded below on $\R^n \setminus \{ 0 \}$. Suppose otherwise. Then $u$ is bounded above on $\R^n \setminus \{ 0 \}$ and the map $r \mapsto m(r)$ is decreasing and $m(r) \to -\infty$ as $r \to \infty$. From the Harnack inequality we deduce that $u(x) \to -\infty$ as $|x|\to \infty$. By the maximum principle, it follows that $u \leq 2 \Phi + k$ for any $k\in \R$. This is obviously a contradiction, as we can let $k \to -\infty$. Thus $u$ is bounded below.

By adding a constant to $u$, we may assume that $\inf_{\R^n \setminus \{ 0 \}} u = 0$. Then $m(r) \to 0$ as $r \to \infty$, and using the Harnack inequality again we deduce that $u(x) \to 0$ as $|x| \to \infty$. The maximum principle immediately yields that $(1-\ep) \Phi - \ep \leq u \leq (1+\ep)\Phi + \ep$ for every $\ep > 0$. Now we let $\ep \to 0$ to deduce that $u \equiv \Phi$.

Case (ii)(b): $\anom(F) = 0$ and $u(x) = \Phi(x) + O(1)$ as $x\to 0$.
We claim that
\begin{equation}\label{eq:liouville-case2b-wts}
\bar \delta := \limsup_{r\to \infty} \bar \rho (r) \geq 1.
\end{equation}
Suppose on the contrary that $\bar \delta < 1$, and select $\delta_1 >
0$ such that $\bar \delta < \delta_1 < 1$. Then for every $k \in \R$
and every $s>1$ sufficiently large, we have
\begin{equation*}
u \geq \delta_1 \Phi + k \quad \mbox{on} \ \partial B_s \cup \partial B_{1/s}.
\end{equation*}
Hence $u \geq \delta_1 \Phi + k$ in $B_s \setminus B_{1/s}$ by the
maximum principle. Sending $s\to \infty$ and then $k\to \infty$ yields
$u \equiv + \infty$. This contradiction establishes
\EQ{liouville-case2b-wts}. A similar argument verifies that $\delta :=
\liminf_{r \to \infty} \rho(r) \leq 1$, and the Harnack inequality
implies that $\delta > 0$. By inspecting the alternatives in
\THM{singularity-infinity}, we see that
\begin{equation*}
a\Phi - C \leq u \leq a\Phi +C \quad \mbox{in} \ \R^n \setminus \{ 0 \},
\end{equation*}
for some $a>0$. Since $\delta \leq 1 \leq \bar \delta$, we must have
$a=1$. By adding a constant to $u$, we may suppose that
$\max_{\partial B_1} (u-\Phi) = 0$. By the maximum principle, for
every $0 < c < 1$ we have
\begin{equation*}
u \leq \Phi + c( \Phi - \min_{\partial B_1} \Phi ) \quad \mbox{in} \
B_1 \setminus \{ 0 \}.
\end{equation*}
and
\begin{equation*}
u \leq \Phi - c( \Phi - \max_{\partial B_1} \Phi)  \quad \mbox{in} \
\R^n \setminus B_1.
\end{equation*}
Sending $c \to 0$ we find $u \leq \Phi$ in $\R^n$. Since
$\max_{\partial B_1} (u-\Phi) = 0$, the strong maximum principle
implies that $u \equiv \Phi$.

Case (iii): $\anom(\tilde F) \geq 0$ and $u(x) = - \tilde\Phi(x) + O(1)$ as $x \to 0$. We may repeat our arguments in case (ii) above, or simply apply them to $-u$ and the dual operator $\tilde F$, to deduce that $u \equiv -\tilde \Phi$.

Case (iv): $\anom(F) < 0$, $u(0) = 0$, and $\lim_{x\to 0} u(x)  /
\Phi(x) = 1$. By the maximum principle, for every $r>0$ and $\ep > 0$
we have
\begin{equation*}
u + \ep \geq \bar\rho(r) \Phi \quad \mbox{in} \ B_r \setminus \{ 0 \}.
\end{equation*}
Thus $u \geq \bar\rho(r) \Phi$ in $B_r \setminus \{ 0 \}$ for any $r>
0$. Thus $\bar\rho(r) = \sup_{B_r\setminus \{ 0 \}} u/\Phi$, and so
$r\mapsto \bar\rho(r)$ is increasing. According to our assumption
regarding the behavior of $u$ near the origin, we must have
$\bar\rho(r) \geq 1$ for all $r>0$. A similar argument ensures that
$\rho(r) \leq 1$ for all $r> 0$, and that $r\mapsto \rho(r)$ is
decreasing.

In particular, we deduce that $u$ is unbounded from below at infinity,
and hence bounded from above in $\R^n$. Moreover, it is clear from
$\limsup_{r\to\infty} \bar\rho(r) > 0$ that alternative (iv) holds in
\THM{singularity-infinity}, and since $\rho \leq 1 \leq \bar \rho$ we
have
\begin{equation*}
\Phi - C \leq u \leq \Phi + C \quad \mbox{in} \ \R^n \setminus B_1.
\end{equation*}
The monotonicity of $\rho$ and $\bar \rho$ now immediately imply that
$\rho \equiv \bar \rho \equiv 1$, and thus $u \equiv \Phi$.

Case (v): $\anom(\tilde F) < 0$, $u(0) = 0$ and $\lim_{x\to 0} -u(x) / \tilde \Phi(x) = 1$. We may apply the result we have proven in case (iv) to $-u$ and $\tilde F$ to deduce that $u\equiv - \tilde \Phi$.
\end{proof}

\begin{proof}[Proof of \THM{fundysol}]
The theorem is immediately obtained by appealing to \PROP{fundysol} and \THM{liouville}.
\end{proof}

\section{Applications to stochastic differential games} \label{sec:stoch}

In this section we give an interpretation of the scaling exponent $\anom(F)$ in terms of two-player stochastic differential games. In particular, we generalize the well-known fact that Brownian motion is recurrent in dimension $n = 2$ and transient in dimensions $n \geq 3$. For a review of the connection between viscosity solutions of second-order elliptic and parabolic equations and stochastic differential games, we refer to Fleming and Souganidis \cite{Fleming:1989} and Kovats \cite{Kovats:2009}.

Let us briefly describe the probabilistic setting (see \cite{Kovats:2009} for more details). We are given a probability space $(\Omega,\mathcal F, \mathbb{P})$, a filtration of $\sigma$-algebras $\{ \mathcal F_t \}_{t\geq 0}$ which is complete with respect to $(\mathcal F, \mathbb{P})$, and a $d$-dimensional Weiner process $\{ W_t\}_{t\geq 0}$ adapted to $\mathcal{F}_t$. Also given are compact metric spaces $A$ and $B$, which are the \emph{control sets} for Players I and II, respectively, and a function $\sigma : A \times B \to \M^{n\times d}$.

We are interested in a random process $\{ X_t \}_{t\geq 0}$ governed by the stochastic differential equation
\begin{equation}\label{eq:SDE}
\left\{ \begin{aligned}
& dX_t = \sigma(a_t,b_t) dW_t, \\
& X_0 = x \in \R^n.
\end{aligned} \right.
\end{equation}
Here $a_t$ and $b_t$ are $A$ and $B$-valued $\mathcal F_t$-progressively measurable stochastic processes, called the \emph{admissible control processes} for Players I and II, respectively. The set of admissible control processes for Player I is denoted by $\mathcal M$, and for Player II is denoted by $\mathcal N$. Here we do not distinguish between controls $\{ a_t \}, \{ \tilde a_t \} \in \mathcal M$ for which
$
\Prob{ a_t = \tilde a_t \ \mbox{for almost every} \ t \geq 0} = 1.
$

An \emph{admissible strategy for Player I} is a mapping $\gamma : \mathcal{N} \to \mathcal M$, and similarly an \emph{admissible strategy for Player II} is a mapping $\theta: \mathcal M \to \mathcal N$. We denote the set of admissible strategies for Players I and II by $\Gamma$ and $\Theta$, respectively.

For each $r > 0$, we let the random variable $\tau_{r} = \tau_{x,a,b,r}$ denote the first time the process $X_t$ hits the sphere $\partial B_r$. Similarly, we define $\tau_{0} = \tau_{x,a,b,0}$ to be the first time the process $X_t$ touches the origin, and also denote $\tau_{\infty} = \infty$.

We first consider a game played in the annulus $B_R \setminus B_r$, for $0< r < |x| < R \leq \infty$, and for which the \emph{payoff functional} is the map $J = J_{x,r,R}: \mathcal{M} \times \mathcal{N} \to \R$ given by
\begin{equation*}
J_{x,r,R} \left[ a_ t, b_t\right]  : = \Prob{ \tau_{x,a,b,r} < \tau_{x,a,b,R}}.
\end{equation*}
Player I wishes to maximize the payoff and Player II wishes to minimize it. Thus, Player I wishes the process to exit the annulus $B_R \setminus \bar B_r$ on the inner boundary $\partial B_r$ while Player II tries to force the process to exit on the outer boundary $\partial B_R$.

We define the \emph{upper value} of the game by
\begin{equation*}
v_{r,R}^+(x) : = \sup_{\gamma\in \Gamma} \inf_{b \in \mathcal N} J_{x,r,R} \left[ \gamma(b) , b\right],
\end{equation*}
and the \emph{lower value} of the game by
\begin{equation*}
v_{r,R}^-(x) : = \inf_{\theta \in \Theta} \sup_{a \in \mathcal M}  J_{x,r,R} \left[ a ,\Theta(a) \right].
\end{equation*}
For each $M \in \Sy$, we define the \emph{upper Isaacs operator} by
\begin{equation*}
F^+(M) : = - \min_{a \in A} \max_{b\in B}\left\{  \frac{1}{2} \trace( \sigma(a,b) \sigma(a,b)^T M )\right\},
\end{equation*}
and the \emph{lower Isaacs operator} by
\begin{equation*}
F^-(M) : = - \max_{b\in B} \min_{a \in A} \left\{  \frac{1}{2} \trace( \sigma(a,b) \sigma(a,b)^T M )\right\}.
\end{equation*}
It is clear that $F^+(M)  \leq F^-(M)$ for all $M\in\Sy$. According to \cite[Theorem 4.3]{Kovats:2009}, we have the following characterization of the upper and lower value functions.
\begin{prop}
Let $0 < r < R < \infty$. The upper value function $v^+_{r,R}$ is the unique viscosity solution of the boundary-value problem
\begin{equation}
\left\{ \begin{aligned}
& F^+(D^2 v) = 0 & \mbox{in} &  \ B_R\setminus \bar B_r, \\
& v = 1 & \mbox{on} & \ \partial B_r, \\
& v = 0 & \mbox{on} & \ \partial B_R. \\
\end{aligned} \right.
\end{equation}
Similarly, the lower value function $v^-_{r,R}$ is the unique viscosity solution of the boundary-value problem
\begin{equation}
\left\{ \begin{aligned}
& F^-(D^2 v) = 0 & \mbox{in} &  \ B_R\setminus \bar B_r, \\
& v = 1 & \mbox{on} & \ \partial B_r, \\
& v = 0 & \mbox{on} & \ \partial B_R. \\
\end{aligned} \right.
\end{equation}
\end{prop}
As $F^+\leq F^-$, it is clear from the maximum principle that $v_{r,R} ^- \leq v_{r,R}^+$ in $B_R \setminus B_r$. We now show that we may use the fundamental solutions of $F^+$ and $F^-$ to estimate the value functions. We only treat $v^+_{r,R}$, since $v^-_{r,R}$ can be estimated similarly.

We let $\Phi$ and $\alpha=\anom(F^+) > -1$ denote the fundamental solution and scaling exponent of $F^+$.
Denote $m(r) := \min_{|x|=r} \Phi(x)$ and $M(r) : = \max_{|x|=r} \Phi(x)$ for each $r> 0$. Fix two radii $0 < r < R$, and notice that, by comparison,
\begin{equation} \label{eq:value-fun-squish}
\frac{ \Phi(x) - M(R)}{M(r) - M(R)} \leq v^+_{r,R}(x) \leq \frac{ \Phi(x) - m(R)}{m(r) - m(R)}, \quad r < |x| < R.
\end{equation}
Let us suppose first that $\alpha \neq 0$. From these inequalities we obtain
\begin{equation*}
\frac{m(1) |x|^{-\alpha} - M(1) R^{-\alpha}}{ M(1) ( r^{-\alpha} - R^{-\alpha}) } \leq v^+_{r,R}(x) \leq \frac{ M(1) |x|^{-\alpha} - m(1) R^{-\alpha} }{ m(1) (r^{-\alpha} - R^{-\alpha}) }, \quad r < |x| < R.
\end{equation*}
Suppose now that $\alpha > 0$, and fix $x \in \R^n \setminus \{ 0 \}$. Then we obtain
\begin{equation*}
\lim_{r\to 0} \frac{ \log v^+_{r,R}(x)}{\log r} = \anom(F),
\end{equation*}
and this limit is uniform in $R$, for large $R$. That is, for any $R> 0$ we have
\begin{equation*}
\anom(F^+) = \lim_{r\to 0} \sup_{\gamma \in \Gamma} \inf_{b\in \mathcal N} \frac{ \log \Prob{\tau_{x,\gamma(b),b,r} < \tau_{x,\gamma(b),b,R}}}{\log r}.
\end{equation*}
We can let $R\to \infty$ to obtain
\begin{equation*}
\anom(F^+) = \lim_{r\to 0} \sup_{\gamma \in \Gamma} \inf_{b\in \mathcal N} \frac{ \log \Prob{\tau_{x,\gamma(b),b,r} < \infty}}{\log r}.
\end{equation*}
That is, for small $r>0$ we have
\begin{equation} \label{eq:goaway}
 \sup_{\gamma \in \Gamma} \inf_{b\in \mathcal N} \Prob{ \tau_{x,\gamma(b),b,r} < \infty} \sim  r^{\anom}.
\end{equation}
We conclude that if $\alpha> 0$, then whichever strategy $\gamma \in \Gamma$ Player I selects, the second player can find a control process $b \in \mathcal{N}$ so that the resulting diffusion process is \emph{transient}. That is, with probability 1, the process $\{ X_t \}$ converges to infinity in the sense that it eventually leaves every bounded set and never returns. To see this, recall that the random process $\{ X_t \}$ eventually must leave any ball since its variance is positive. Furthermore, we see from \EQ{goaway} that the process $\{ X_t \}$ returns to any given ball infinity often with probability zero.

Let us suppose instead that $\alpha = \anom(F^+) < 0$. Then \EQ{value-fun-squish} can be rewritten as
\begin{equation*}
 \frac{ m(1)r^{-\alpha} - M(1) |x|^{-\alpha} }{ m(1) (r^{-\alpha} - R^{-\alpha}) } \leq 1-  v_{r,R}^+(x) \leq \frac{M(1)r^{-\alpha} - m(1) |x|^{-\alpha}}{ M(1) ( r^{-\alpha} - R^{-\alpha}) }, \quad r < |x| < R,
\end{equation*}
and we obtain
\begin{equation*}
\lim_{R\to \infty} \frac{ \log \left(1- v_{r,R}(x)\right)}{\log R} = \anom(F),
\end{equation*}
and this limit is uniform in $r>0$. Sending $r\to 0$ we find that for large $R>0$
\begin{equation*}
 \sup_{\gamma \in \Gamma} \inf_{b\in \mathcal N} \Prob{ \tau_{x,\gamma(b),b,0} < \tau_{x,\gamma(b),b,R}} \sim 1- R^{\anom}.
\end{equation*}
We can let $R\to \infty$ to obtain
\begin{equation*}
\sup_{\gamma \in \Gamma} \inf_{b\in \mathcal N} \Prob{\tau_{x,\gamma(b),b,0} < \infty} = 1.
\end{equation*}
That is, Player I can find a strategy which ensures the process $\{ X_t \}$ returns to the origin almost surely. Therefore the process $\{ X_t \}$ is \emph{recurrent} in a very strong sense.

The case $\alpha = \anom(F^+) = 0$ is a compromise between the two cases discussed above. The estimate \EQ{value-fun-squish} implies
\begin{equation*}
\frac{ m(1) - M(1) + \log R}{\log R - \log r} \leq v^+_{r,R}(x) \leq  \frac{ M(1) - m(1) + \log R}{\log R - \log r}.
\end{equation*}
From these inequalities, we see that $v^+_{r,R} (x) \to 1$ as $R \to \infty$, and $v^+_{r,R}(x) \to 0$ as $r \to 0$. Thus, Player I has a strategy $\gamma$ which ensures that the diffusion $\{ X_t \}$ will almost surely return to every neighborhood of the origin infinitely many times, but Player II may select a control process which ensures that the process $\{ X_t \}$ never touches the origin (almost surely). Since $\anom(-\Delta) = 0$ in dimension $n=2$, this is how Brownian motion behaves in the plane.

\subsection*{Acknowledgements.} This research was conducted in part while the first author was a visitor at Le Centre d'analyse et de math\'ematique sociales (CAMS) in Paris.

\bibliographystyle{plain}
\bibliography{fundysols-bib}

\end{document}